\documentclass[11pt,a4paper]{article}
\usepackage{fancyhdr}
\usepackage[british]{babel}

\usepackage{amssymb, amsmath, amsthm, amsfonts, enumerate,bm,nicefrac}
\usepackage{amsmath,setspace,scalefnt}
\usepackage[usenames,dvipsnames,svgnames,table]{xcolor}
\usepackage{graphicx,tikz,caption,subcaption}
\usetikzlibrary{patterns}
\usetikzlibrary{shapes}
\usepackage{fullpage} 
\usepackage[colorlinks=true,anchorcolor=blue,filecolor=blue,linkcolor=red,urlcolor=blue,citecolor=blue]{hyperref}  

\usepackage{enumitem,verbatim}
\usepackage{multicol}
\usepackage{float}
\usepackage{cleveref}

\AtBeginDocument{%
  \setlength{\abovedisplayskip}{5pt plus 2pt minus 3pt}%
  \setlength{\belowdisplayskip}{5pt plus 2pt minus 3pt}%
  \setlength{\abovedisplayshortskip}{3pt plus 2pt minus 1pt}%
  \setlength{\belowdisplayshortskip}{3pt plus 2pt minus 2pt}%
}

\usepackage[margin=2.2cm]{geometry}

\definecolor{gray}{rgb}{0.25, 0.25, 0.25}

\usepackage{booktabs}  

\counterwithin{figure}{section}

\newtheorem{theorem}{Theorem}[section]
\newtheorem{lemma}[theorem]{Lemma}
\newtheorem{claim}[theorem]{Claim}
\newtheorem{cor}[theorem]{Corollary}

\newtheorem*{observation*}{Observation}

\newtheorem*{problem*}{Problem}

\newtheorem*{question*}{Question}
\newtheorem{innerdefinition}{Definition}
\newenvironment{definition*}
  {
   \innerdefinition}
  {\endinnerdefinition}

\theoremstyle{definition}

\theoremstyle{remark}

\newenvironment{poc}{\begin{proof}[Proof of claim]}{\end{proof}}

\newcommand{\ceil}[1]{\lceil #1\rceil}
\newcommand{\floor}[1]{\lfloor #1\rfloor}

\usepackage{todonotes} 

\newcommand{\eps}{\varepsilon}

\newcommand*{\abs}[1]{\lvert#1\rvert}

\title{An edge-spectral Erd\H{o}s--Stone--Simonovits theorem\\ and its stability\footnote{An earlier version of this paper was announced in August 2025 on arXiv. \url{https://arxiv.org/abs/2508.15271}}}

\author{
Yongtao Li\thanks{Yau Mathematical Sciences Center, Tsinghua University, Beijing, China. Email: 
\url{ytli0921@hnu.edu.cn}. Supported by the Postdoctoral
Science Foundation of China (2026M793384), the Postdoctoral Fellowship Program of China Postdoctoral Science Foundation (GZC20261656) and 
the National Natural Science Foundation of China (12601675).} 
\and
Hong Liu\thanks{Extremal Combinatorics and Probability Group (ECOPRO), Institute for Basic Science (IBS), Daejeon, South Korea. Email: \url{hongliu@ibs.re.kr}. Supported by the Institute for Basic Science (IBS-R029-C4).}
\and
Shengtong Zhang\thanks{Department of Mathematics, Stanford University, CA, USA. Email: \url{stzh1555@stanford.edu}. Supported by the Craig Franklin Fellowship in Mathematics at Stanford University.}
}
\date{}

\begin{document}
\maketitle

\vspace{-8mm}
\begin{abstract}
We study the extremal problem that relates the spectral radius $\lambda (G)$ of an $F$-free graph $G$ to its number of edges. 
Firstly, we prove that for any graph $F$ with chromatic number $\chi (F)=r+1\ge 3$, if $G$ is an $F$-free graph with $m$ edges, then $\lambda^2(G)\le {(1-\frac{1}{r} + o(1))2m}$. 
This provides a unified extension of both the Erd\H{o}s--Stone--Simonovits theorem and its vertex-spectral extension due to Nikiforov, and confirms a conjecture proposed by Li, Liu and Feng.

We also establish the corresponding edge-spectral stability, showing that if $G$ is an $F$-free graph with $m$ edges with   $\lambda^2(G) \ge (1- \frac{1}{r} - o(1))2m$, then $G$ differs from a complete bipartite graph by $o(m)$ edges when $r=2$, and $G$ differs from an $r$-partite Tur\'{a}n graph by $o(m)$ edges when $r\ge 3$. 
This extends both the classical Erd\H{o}s--Simonovits stability and the Nikiforov vertex-spectral stability. 
This edge-spectral stability theorem is the key ingredient in our companion work on the Brualdi--Hoffman--Tur\'{a}n problem for color-critical forbidden graphs. 

As an application of our method, we improve a result of Zhai, Lin and Shu by showing that if $\lambda (G)>\sqrt{m}$, then there exist two vertices in $G$ that have at least $\frac{1}{2}\sqrt{m} - O(1)$ common neighbors. The constant $\frac{1}{2}$ is best possible as witnessed by a random construction. 
\end{abstract}

{{\bf Keywords:}   Extremal combinatorics; Erd\H{o}s--Stone--Simonovits; Spectral radius; Stability. }

{{\bf 2020 Mathematics Subject Classification:}  05C50; 05C35.}

\section{Introduction}

Extremal graph theory investigates how local structural constraints affect global graph parameters, and it has deep connections to other mathematical disciplines. The field originated with foundational questions about how dense a graph can be while avoiding certain forbidden subgraphs, giving rise to its most classical theme: Tur\'{a}n-type problems. The solution to such problems often involves both precise bounds and the characterization of graphs that achieve them; see  \cite{Bollobas78,FS13,Sim2013,GTAC}. 
The stability results describe how nearly extremal graphs -- those close to the maximum size -- must structurally resemble the extremal maximizer; see \cite{Sim1966}. 

A graph $G$ is called $F$-free if it does not contain a subgraph isomorphic to $F$. A well-known theorem of Tur\'{a}n \cite{Tur1941} 
states that every $n$-vertex $K_{r+1}$-free graph $G$ satisfies $e(G)\le \left(1-\frac{1}{r} \right) \frac{n^2}{2}$; 
see \cite[p. 294]{Bollobas78} and \cite[p. 269]{AZ2014}. 
Since then, there have been various extensions and generalizations of Tur\'{a}n's theorem; see, e.g., \cite{BT1981,Bon1983}.  
The \emph{Tur\'{a}n number} $\mathrm{ex}(n,F)$ is defined as the maximum number of edges in an $n$-vertex $F$-free graph. 
A celebrated extension of Tur\'{a}n's theorem is due to 
 the Erd\H{o}s--Stone--Simonovits theorem \cite{ES46,ES66}, which states that the asymptotics of $\mathrm{ex}(n,F)$ is determined by the chromatic number of $F$. More precisely, for any graph $F$ with chromatic number $\chi (F)=r+1$, we have $\mathrm{ex}(n,F) = \left(  1-\frac{1}{r} + o(1)\right) \frac{n^2}{2}$.

In this paper, we study the spectral extensions of the Erd\H{o}s--Stone--Simonovits theorem.  Spectral graph theory aims to apply the eigenvalues and ranks of matrices associated with graphs to find the structural information of graphs. 
We refer to \cite{Huang2019,JTYZZ2021} for recent breakthroughs and applications. 
The adjacency matrix of an $n$-vertex graph $G$ is defined as $A_G=[a_{i,j}]_{i,j=1}^n$, where $a_{i,j}=1$ if $ij\in E(G)$, and $a_{i,j}=0$ otherwise. Observe that $A_G$ is a real and symmetric nonnegative matrix.  
Let $\lambda (G)$ be the spectral radius of $G$, which is defined as the maximum absolute value of the eigenvalues of $A_G$. 
Spectral extremal graph theory seeks the connections between graph properties and eigenvalues of the associated matrices. In this direction, many classical results in extremal graph theory have been generalized to spectral statements. 
For instance, Wilf \cite{Wil1986} proved a spectral extension of Tur\'{a}n's theorem, showing that if $G$ is an $n$-vertex $K_{r+1}$-free graph, then $\lambda (G)\le \left(1- \frac{1}{r} \right)n$. This was generalized by Nikiforov  \cite{Niki2009cpc} in 2009, who established the following spectral counterpart of the Erd\H{o}s--Stone--Simonovits theorem.

\begin{theorem}[Nikiforov~\cite{Niki2009cpc}] 
 \label{thm-spec-ESS} 
If $\chi (F)=r+1$ and $G$ is an $n$-vertex $F$-free graph, then 
\begin{equation*} 
  \lambda(G) \le 
\left( 1-\frac{1}{r} + o(1)\right) n,  
\end{equation*}
where the $o(1)$ term tends
to zero as $n \to \infty$.  
\end{theorem}

\subsection{The Brualdi--Hoffman--Tur\'{a}n problem}

It is well-known that  for every graph $G$ with $m$ edges,\footnote{When we consider the extremal graph problem on a graph in terms of the number of edges, we usually ignore isolated vertices if there are no confusions.} we have
$\lambda (G)< \sqrt{2m}$. 
This bound follows from e.g.~$\lambda^2(G)< \mathrm{Tr}(A^2_G) = \sum_{v\in V}d(v)=2m$.    
This bound can be improved if $G$ has small clique number. In 2002, Nikiforov \cite{Niki2002} proved that if $G$ is $K_{r+1}$-free with $m$ edges, then 
\begin{equation}
    \label{eq-Niki-2002-cpc}
    \lambda (G)\le \sqrt{\Big(1- \frac{1}{r} \Big)2m}.
\end{equation}
Moreover, the extremal graphs attaining the equality are complete bipartite graphs for $r=2$, or regular complete $r$-partite graphs for every $r\ge 3$; see \cite{Niki2006-walks,Niki2009jctb}. 
In particular, if $G$ is an $m$-edge graph with $\lambda (G)> \sqrt{m}$, then $G$ contains a triangle. 
This result was initially studied by Nosal \cite{Nos1970} in 1970, and we refer to \cite{Niki2021, NZ2021, LLZ2024-book-qua} for recent extensions. 

In contrast with Wilf's theorem, which connects the spectral radius of $G$ with its number of vertices, 
the Nikiforov result (\ref{eq-Niki-2002-cpc}) bounds  the spectral radius using the number of edges. Note that the edge-version in~(\ref{eq-Niki-2002-cpc}) extends both Tur\'an's theorem and Wilf's theorem by  the fact $\lambda (G)\ge \frac{2m}{n}$. 
It is worth pointing out that the spectral assumption in terms of the number of edges is more general and versatile: the edge-spectral version not only implies the vertex-spectral version, but also applies to sparse graphs with any edge density. As such, there has been great interest in studying the maximum spectral radius for $F$-free graphs with a given number of edges. 

One of the first results of this kind is due to Brualdi and Hoffman \cite{BH1985}, who proved in 1985 that any $m$-edge graph $G$ satisfies $\lambda (G)\le k-1$ if $m\le {k \choose 2}$. 
 Stanley \cite{Sta1987} extended this by showing that 
 $\lambda (G)\le \frac{1}{2} 
 \left(\sqrt{8m+1} -1 \right).$
 The bound is optimal for complete graphs;  
see~\cite{HSF2001,Hong1988,Hong1998,Niki2002,ZC2005} for more related results.  
Motivated by the results of Brualdi,  Hoffman and Tur\'{a}n mentioned above, we consider the following extremal spectral graph problem. 

\begin{problem*}[Brualdi--Hoffman--Tur\'{a}n type problem] 
Given a graph $F$, 
   what is the maximum spectral radius over all $F$-free graphs with $m$ edges without isolated vertices?  
\end{problem*}

This line of investigation has emerged as an active research area in extremal spectral graph theory in recent years.  
For example, we refer to \cite{Niki2002,BN2007jctb,Niki2009jctb,Zhang2024,CSZ2024} for the case when $F$ is a clique; 
\cite{Niki2009,ZLS2021,ZS2022dm,LZS2024} when $F$ is a cycle; 
\cite{LLP2023,YLP2024,CY2024,LZZ2024} when $F$ is a fan; 
see also~\cite{Niki2021,NZ2021b,LLZ2024-book-qua,LLZ2025}.

The goal of this paper is to systematically investigate the Brualdi--Hoffman--Tur\'{a}n type problem for any general graph $F$ and establish the corresponding stability result.

\subsection{Extension of the Erd\H{o}s--Stone--Simonovits theorem}

Our first result resolves the Brualdi--Hoffman--Tur\'{a}n  problem asymptotically for any graph $F$. 

\begin{theorem} \label{thm-edge-spectral-ESS}  
If $\chi (F) = r+1 \ge 3$ and $G$ is an $F$-free graph with $m$ edges, then 
$$\lambda^2 (G)\le \left(1-\frac{1}{r} + o(1)\right)2m,$$
where the $o(1)$ term tends
to zero as $m\to \infty$.  
\end{theorem} 

The bound in Theorem \ref{thm-edge-spectral-ESS} is tight 
when $G$ is chosen as an $r$-partite Tur\'{a}n graph with $m$ edges. 
Theorem \ref{thm-edge-spectral-ESS} confirms a conjecture of Li, Liu and Feng \cite[Conjecture 1.31]{LLF2022}. Furthermore, combined with the Rayleigh bound $\lambda(G) \geq \frac{2m}{n}$, Theorem \ref{thm-edge-spectral-ESS}  implies both the Erd\H{o}s--Stone--Simonovits theorem and its vertex-spectral version in Theorem \ref{thm-spec-ESS}.

The condition $\chi (F)\ge 3$ in Theorem \ref{thm-edge-spectral-ESS} 
is necessary, since Theorem \ref{thm-edge-spectral-ESS} can fail when $F$ is a non-star bipartite graph. 
For example, taking $F=K_{2,t}$,  
we can see that $K_{1,m}$ is $F$-free and $\lambda (K_{1,m})= \sqrt{m}$.  
Our companion paper \cite{LLZ2025-part-2-critical}
treats the bipartite case. It proves the tight bound
$\lambda(G)\le \sqrt{e(G)}+O_F(1)$ for every bipartite
graph $F$ and every $F$-free graph $G$. 

In 2006, 
Nikiforov \cite{Niki2006-walks} further generalized (\ref{eq-Niki-2002-cpc}). 
An $\ell$-walk is a sequence of $\ell$ 
vertices $v_1,v_2,\ldots ,v_{\ell}$ of $G$ such that 
$v_i$ is adjacent to $v_{i+1}$ for any $i=1,\ldots ,\ell -1$. 
Let $w_{\ell}(G)$ denote the number of $\ell$-walks in $G$. For example, we have $w_1(G)=n$, $w_2(G)=2m$ and
$w_3(G) = \sum_{v\in V}d_v^2$.

\begin{theorem}[Nikiforov \cite{Niki2006-walks}] \label{thm-Niki-2006}
    If $r\ge 2$ and $G$ is a $K_{r+1}$-free graph, then for every $\ell \ge 1$, 
    $$\lambda^\ell (G) \leq \left( 1- \frac{1}{r}\right)\cdot w_\ell(G).$$
\end{theorem}
Note that the case $\ell=1$ above reduces to Wilf's result,  and the case $\ell=2$ gives the bound in~\eqref{eq-Niki-2002-cpc}. 
Very recently, Chao and Yu \cite{Chao-Yu-2024} extended Theorem \ref{thm-Niki-2006} 
using the entropy method, connecting powers of the spectral radius to homomorphisms of trees in $K_{r+1}$-free graphs. 

The framework of our proof of Theorem \ref{thm-edge-spectral-ESS} can be adapted to prove the following result, which generalizes Theorem \ref{thm-Niki-2006} to $F$-free graphs for an arbitrary graph $F$. Theorem \ref{thm:walk} is a unified extension of both Theorems \ref{thm-spec-ESS} and \ref{thm-edge-spectral-ESS}. 

\begin{theorem}
    \label{thm:walk} 
    If $\chi (F) = r+1\ge 3$ and $G$ is an $F$-free  graph, then 
    $$\lambda^{\ell} (G) \leq \left(1- \frac{1}{r} + o_{\ell, F}(1)\right) \cdot w_\ell(G), $$
    where the $o_{\ell,F}(1)$ term tends to zero as $w_\ell(G)\to \infty$.
\end{theorem}

\subsection{An edge-spectral stability result}

In 1966, 
Erd\H{o}s \cite{Erd1966Sta1,Erd1966Sta2} and Simonovits \cite{Sim1966} proved a stability phenomenon for Tur\'{a}n-type problems. Given two graphs $G$ and $H$, the \emph{edit distance} $d(G, H)$ is the minimum number of edges we need to change (add/remove) to transform one graph into the other. The Erd\H{o}s--Simonovits stability states that for any graph $F$ with $\chi (F)=r+1$, any $n$-vertex $F$-free graph $G$ with  
$\left( 1-\frac{1}{r}-o(1) \right)\frac{n^2}{2}$ edges must be at edit distance $o(n^2)$ from the balanced complete $r$-partite Tur\'{a}n graph $T_{n,r}$. 
We refer the reader to F\"{u}redi \cite{Fur2015} for a short proof.

The following vertex-spectral stability can be deduced from a result of Nikiforov \cite{Niki2009jgt}.

\begin{theorem}[Nikiforov \cite{Niki2009jgt}]  
\label{thm-niki-vertex-stability}
Let $F$ be a graph with $\chi (F)=r+1\ge 3$. 
For every $\varepsilon >0$, there exist $\delta >0$ 
and $n_0$ such that 
if  $G$ is an $F$-free graph on $n\ge n_0$ vertices  with $\lambda (G) \ge \left(1- \frac{1}{r} -\delta \right)n,$
then $G$ differs from $T_{n,r}$ in at most $\varepsilon n^2$ edges. 
\end{theorem}

Theorem \ref{thm-niki-vertex-stability} provides an effective approach to spectral extremal problems when forbidding a non-bipartite graph $F$, and has found several applications in recent years.

In this section, we show a similar stability result for the edge-spectral Tur\'{a}n problem. 
For two disjoint vertex sets $A,B$, 
we write $K_{A,B}$ for the complete bipartite graph on the parts $A$ and $B$. 
For a vertex set $C$, we write $T_{C,r}$ 
for an $r$-partite Tur\'{a}n graph on the vertex set $C$. 

\begin{theorem} 
\label{thm-edge-spectral-stability}
Let $F$ be a graph with $\chi (F)=r+1\ge 3$.  
For every $\varepsilon > 0 $, 
there exist $\delta>0$ and $m_0$ such that if $G$ is an $F$-free graph with $m\ge m_0$ edges and $\lambda^2 (G)\ge (1- \frac{1}{r} - \delta) 2m$, then 
\begin{itemize}
\item[\rm (a)] for $r = 2$, there exist disjoint vertex sets $A, B \subseteq V(G)$ such that $d(G, K_{A, B}) \leq \varepsilon m$.

\item[\rm (b)]
 for $r \geq 3$, there exists a Tur\'{a}n graph $T_{C, r}$ on $C \subseteq V(G)$ such that $d(G, T_{C, r}) \leq \varepsilon m$. 
 \end{itemize}
\end{theorem}

Similarly, combined with the Rayleigh bound $\lambda(G)\ge \frac{2m}{n}$, the edge-spectral stability in Theorem \ref{thm-edge-spectral-stability} generalizes both the Erd\H{o}s--Simonovits stability and the Nikiforov vertex-spectral stability.
We remark that when $r = 2$, our edge-spectral version in Theorem \ref{thm-edge-spectral-stability} differs substantially from the other two stabilities, as the graph in Theorem \ref{thm-edge-spectral-stability} can be close to any complete bipartite graph instead of just the balanced complete bipartite graph. Moreover, we emphasize that the disjoint sets $A$ and $B$ do not necessarily form a partition of the vertex set of $G$, e.g., when $G$ contains some isolated vertices or `bad' vertices with sparse edges. In addition, an analogue of the edge-spectral stability in the bipartite case $F$ can be deduced from Theorem \ref{thm-edge-spectral-stability} by adding an additional edge to make $F$ tripartite; see also \cite[Corollary 2.2]{LLZ2025-part-2-critical}. 

In the next subsection, we give an application of the edge-spectral stability theorem. We note that Theorem~\ref{thm-edge-spectral-stability} also provides a useful tool for solving edge-spectral Tur\'{a}n-type problems. In the companion paper \cite{LLZ2025-part-2-critical}, we use it to develop the \emph{edge-spectral stability method}. This method yields sharp results on the Brualdi--Hoffman--Tur\'{a}n problem for color-critical graphs and for several almost-bipartite graphs. Moreover, the lemmas and theorems of this paper have been used as inputs in several subsequent works, including \texttt{arXiv:2603.14964}, \texttt{2608.16127} and \texttt{2609.13666}. 

\subsection{Application: maximizing common neighbors}

The classical Tur\'{a}n result of F\"{u}redi \cite{Fur1996jcta} states that $\mathrm{ex}(n,K_{2,t+1})=\frac{1}{2}\sqrt{t}n^{3/2} +O(n^{4/3})$. 
In contrast, when forbidding bipartite graphs, the edge-spectral extremal problem has a rather different behavior. 
In 2021, Zhai, Lin and Shu \cite{ZLS2021} proved the following spectral result.  We refer the reader to \cite{Niki2009} for the case $t=1$ and to \cite{ZS2022dm,NZ2021b,LLZ2024-book-qua} for recent results in this direction.

\begin{theorem}[Zhai--Lin--Shu \cite{ZLS2021}]
\label{thm-ZLS-common}
Let $t\ge 2$ and $m\ge 16t^2$ be integers. 
If $G$ is a $K_{2,t+1}$-free graph with $m$ edges, then $\lambda (G)\le \sqrt{m}$, with equality if and only if $G$ is a star.  
\end{theorem}

We point out here that the quadratic dependence on $m$ in terms of $t$ is necessary. Moreover, the above result implies that in any graph $G$ with $\lambda (G)> \sqrt{m}$, we can find a copy of $K_{2,t}$ with $t=\frac{1}{4} \sqrt{m}$. On the other hand, there are many examples showing that such a graph $G$ cannot contain a copy of $K_{2,t}$ with $t=\sqrt{m}$. 
Using the edge-spectral stability~Theorem \ref{thm-edge-spectral-stability}, we can obtain the following improvement with the optimal bound on $t$.

\begin{theorem} \label{thm-common}
    If $G$ is an $m$-edge graph with 
    $\lambda(G) > \sqrt{m}$, then $G$ contains a copy of $K_{2, t}$ with $t \ge \frac{1}{2}\sqrt{m}  - O(1)$. Moreover, the constant $\frac{1}{2}$ is  best possible, as witnessed by random graphs. 
\end{theorem}

\paragraph{Our approach.} 
To prove Theorem \ref{thm-edge-spectral-ESS}, we use the removal lemma to pass from $G$ to a $K_{r+1}$-free subgraph $G'$ by deleting $o(|V(G)|^2)$ edges. If $G$ is dense, we would be done with the help of Nikiforov's result as the removal of these $o(|V(G)|^2)$ edges contributes $o(\sqrt{m})$ to the spectral radius of $G$. Our key novel step in the proof of Theorem \ref{thm-edge-spectral-ESS} is to make a reduction to the dense case. To this end,  
we first show that if $G$ is a minimal counterexample, then the weight of every edge $uv$ in $G$ with respect to the Perron--Frobenius vector is large: $x_ux_v=\Omega (m^{-1/2})$ (Lemma \ref{lem-product-entry}). Secondly, 
we bound the $\infty$-norm of the unit Perron--Frobenius eigenvector of $G$ by showing that $\max_{v\in V}\{x_v\}=O(m^{-1/4})$ (Lemma \ref{lem:max}). Combining these two bounds, we can then deduce that the Perron--Frobenius vector of any possible minimal counterexample must be flat, i.e.~each of its coordinates is of order $\Theta(m^{-1/4})$. As a result, this minimal counterexample $G$ must be dense: $|V(G)|=O(\sqrt{m})$. This approach can be extended to prove Theorem \ref{thm:walk}. In contrast to the edge-spectral setting, we will prove a reduction so that $|V(G)|=O(\lambda (G))$ in the walk-spectral setting for Theorem \ref{thm:walk}.  

For Theorem \ref{thm-edge-spectral-stability}, we deliver the proof in three stages. Firstly, we deal with the case when we forbid triangles (Theorem \ref{thm-stability-triangle}). Using the second moment of the spectral radius of $G$, we can remove $o(m)$ edges from $G$ to obtain a bipartite subgraph $G'$ (Lemma \ref{lem-bipartite}). We then show that bipartite graphs with spectral radius close to $\sqrt{m}$ must be close to complete via truncation on its Perron--Frobenius eigenvector (Lemma \ref{lem:close-to-Kab}). Secondly, when we forbid a clique $K_{r+1}$, where $r\ge 3$ (Theorem \ref{thm-stability-clique-4}), we exploit the slackness from the use of the Motzkin--Straus inequality in Nikiforov's bound (\ref{eq-Niki-2002-cpc}). 
This is the bulk of our proof for Theorem \ref{thm-edge-spectral-stability}. We analyze the Perron--Frobenius eigenvector of $G$ and use the slackness to show that most of the edges in $G$ have Perron--Frobenius weight around $\Theta(m^{-1/2})$. From this fact, we are able to locate a subset $V'$ of vertices in which every vertex has weight about $\Theta(m^{-1/4})$. Furthermore, most of the edges lie within $V'$. As a result, the induced subgraph $G'$ on $V'$ satisfies $d(G,G')=o(m)$ and $|V'|\le (1+o(1))\sqrt{2e(G')/(1-\frac{1}{r})}$, which implies $\lambda (G')\ge \left(1- \frac{1}{r} - o(1) \right)|V'|$ (Lemma \ref{lem:nikiforov-stability-clique}). We can then apply the vertex-spectral stability to $G'$ to finish the proof. Finally, when forbidding a general graph $F$ with $\chi (F)=r+1$, we first remove few ``bad'' edges of $G$ that have small contributions to $\lambda (G)$ to obtain a large subgraph $G'$ on vertex set $V'$ such that $x_i'x_j'=\Omega (m^{-1/2})$ for every $ij\in E(G')$ (Lemma \ref{lem:removing-bad-edges}). If $r\ge 3$, we can apply Lemma \ref{lem:max} to show that $\max_{v\in V'}\{x'_v\}=O(m^{-1/4})$, which implies $|V'|=O(\sqrt{m})$.  
Then we can use the removal lemma to pass to a $K_{r+1}$-free graph $G''$ by removing $o(|V'|^2)=o(m)$ edges of $G'$. This allows us to reduce the general case to the clique case that we have established previously. For the case when $r=2$, Lemma \ref{lem:max} cannot be applied and 
it is possible that $\max_{v\in V'}\{x_v'\}=\Omega(m^{-1/4})$. 
We need to treat this case separately (Lemma \ref{lem:condition-max}). Here we show that $G'$ contains a bipartite subgraph $G''$ with large spectral radius, and we can use Lemma \ref{lem:close-to-Kab} to finish the proof. 

For Theorem \ref{thm-common}, 
{the leading constant $\tfrac12$ is best possible: for any $\varepsilon>0$, the random graph $G(n,\tfrac12+\varepsilon)$ has, with high probability, spectral radius larger than $\sqrt{m}$ yet every two vertices share at most $(\tfrac12+6\varepsilon)\sqrt{m}$ common neighbors, so the coefficient $\tfrac12$ cannot be improved. We defer the details to Section \ref{sec:common}.} 
To prove the lower bound, 
we first show that under the hypothesis  
$\lambda(G)> \sqrt{m}$, it follows that 
$t+ 2m\cdot (\max_{i\in V} x_i)^4 >\frac{1}{2} \sqrt{m}$, where $t$ is the maximum codegree. If $\max_{i\in V} \{x_i\} = O(m^{-1/4})$, then $t\ge \frac{1}{2}\sqrt{m} - O(1)$. If $\max_{i\in V}\{x_i\} =\Omega (m^{-1/4})$, then we can use the edge-spectral stability (Lemma \ref{lem:condition-max}) to get two disjoint vertex sets $A,B\subseteq V(G)$ such that $d(G,K_{A,B})\le \varepsilon m$. We may assume that $|A|\le |B|$. As $\lambda(G)> \sqrt{m}$, we can use a structural result (Lemma \ref{lem-Delta}) to show that $|A|\ge 2$. Then by a standard average argument, there exist two vertices of $A$ that have at least $(1-o(1))\sqrt{m}$ common neighbors in $B$, as desired.

\paragraph{Organization.}
In Section \ref{sec:prelim}, we provide some preliminaries. 
In Section \ref{sec:ESS}, 
we give the proofs of Theorems \ref{thm-edge-spectral-ESS} and \ref{thm:walk}. In Section \ref{sec:edge-stability}, we present the proof of Theorem \ref{thm-edge-spectral-stability}. 
In Section \ref{sec:common}, we give the proof of Theorem \ref{thm-common} and in Section \ref{sec:concluding} we provide some concluding remarks. 
In Appendix \ref{sec-A}, we give an alternative proof of Theorem \ref{thm-edge-spectral-stability} in the case when the forbidden graph $F$ is a clique.

\paragraph{Notation.} 
We write $G=(V,E)$ for a simple and undirected graph with vertex set $V=\{v_1,\ldots ,v_n\}$ and edge set $E=\{e_1,\ldots ,e_m\}$, where we admit 
$n=|V|$ and $m=|E|$. 
We write $uv$ for an edge if $u$ and $v$ are adjacent.  
The degree of a vertex $v\in V(G)$ is denoted by 
$d_v$, and the set of neighbors of $v$ is denoted by $N(v)$. 
The maximum degree of $G$ is denoted by 
$\Delta (G)$. 
Let $w_{\ell} (G)$ be the number of walks in $G$ consisting of $\ell$ vertices (we call it an $\ell$-walk). 
Let $\bm{1}_v$ be the characteristic vector of a vertex $v$.  
For $U\subseteq V(G)$, 
we write $d_U(v)$ for the number of vertices of $U$ that are adjacent to the vertex $v$.  
For two disjoint sets $U,W\subseteq V(G)$, 
we write $e(U)$ for the number of edges with endpoints in $U$, and $e(U,W)$ for the number of edges between $U$ and $W$. 
The edit distance between two graphs $G$ and $H$ is denoted by $d(G,H)$, which is the minimum number of edges we need to change (delete or add) to transform $G$ into $H$.  

The adjacency matrix $A_G$ 
of a graph $G$ is a symmetric matrix of dimension $|V|\times |V|$ with $a_{i,j}=1$ if and only if $ij\in E$, and $a_{i,j}=0$ otherwise. 
Since $A_G$ is real and symmetric, 
all eigenvalues of $A_G$ are real and can be reordered as $\lambda_1\ge \lambda_2\ge \cdots \ge \lambda_n$.  The eigenvalues of a graph $G$ are defined as the eigenvalues of the matrix $A_G$. 
The spectral radius $\lambda (G)$ is defined to be the maximum absolute value of eigenvalues of $G$. By the Perron--Frobenius theorem, 
the spectral radius of a nonnegative matrix is actually a largest eigenvalue. Furthermore, there exists a unit nonnegative eigenvector $\bm{x}\in \mathbb{R}^n$  corresponding to $\lambda (G)$. 
Such a vector is called the Perron--Frobenius eigenvector. For each vertex $v\in V(G)$, 
we write $x_v$ for the coordinate of $\bm{x}$ corresponding to $v$. 
In the setting of graphs, 
the eigen-equation $A_G\bm{x}= \lambda (G) \bm{x}$ is interpreted as $\lambda(G) x_v = \sum_{u\in N(v)} x_u$, and the Rayleigh quotient implies 
$\lambda (G) = 2 \sum_{uv\in E} x_ux_v$. 
Here and in the rest of the paper, we denote by $\sum_{uv \in E}$ the sum over each edge in $E$ once. When $G$ is a bipartite graph with given partite sets $A$ and $B$, we additionally assume that $u \in A$ and $v \in B$.

\section{Preliminary results}\label{sec:prelim}
In this section, we collect a number of results that we will use throughout our proofs. The first one is a result of  Erd\H{o}s, Frankl and R\"{o}dl \cite{EFR1986} which allows us to delete $o(n^2)$ edges and transform an $F$-free graph to a $K_{r+1}$-free graph. 

\begin{lemma}[Erd\H{o}s--Frankl--R\"{o}dl \cite{EFR1986}] \label{lem-EFR}
Let $F$ be a graph with $\chi (F)=r+1$. 
For any $\varepsilon >0$, 
there is $n_0$ such that 
if $n\ge n_0$ and $G$ is an $n$-vertex $F$-free graph, 
then we can remove at most $\varepsilon n^2$ edges from $G$ 
so that the remaining graph is $K_{r+1}$-free.  
\end{lemma}

The following lemma states that if $G$ is a minimal graph with a certain bound on the spectral radius, then the product of the Perron--Frobenius weights of each edge is at least $\Omega (m^{-1/2})$.  

\begin{lemma} \label{lem-product-entry}
Let $r\ge 2$, $\varepsilon>0$, $C>0$ and $\bm{x}=(x_v)_{v\in V(G)}$ be the unit Perron--Frobenius vector of $G$. 
Let $G$ be an $m$-edge graph with 
$  \lambda (G)> \sqrt{(1-\frac{1}{r} + \varepsilon )2m} + C$.  
 If $\lambda(G')\le 
\sqrt{(1-\frac{1}{r} + \varepsilon )2m'} + C$ holds for every subgraph $G'\subsetneq G$ with $m'$ edges, then $x_ux_v > \frac{1}{8\sqrt{m}}$ for every edge $uv \in E(G)$. 
\end{lemma}

\begin{proof}
Suppose on the contrary that
$x_ux_v\le \tfrac18 m^{-1/2}$ for some edge
$uv\in E(G)$. Removing $uv$ yields a subgraph $G'$
with $m'=m-1$ edges. Using the unit vector $\bm{x}$
as a test vector for $G'$, the Rayleigh quotient
gives
\[ \lambda(G')\ge 2\!\sum_{ij\in E(G')}\! x_ix_j
   =2\!\sum_{ij\in E(G)}\! x_ix_j-2x_ux_v
   \ge \lambda(G)-\frac{1}{4\sqrt{m}}. \]
{Write $a:=2\bigl(1-\tfrac1r+\varepsilon\bigr)$. Since
$r\ge2$, we have $a\ge 1+2\varepsilon>1$. Hence
\[ \sqrt{am}-\sqrt{am'}
   =\frac{a}{\sqrt{am}+\sqrt{a(m-1)}}
   =\frac{\sqrt a}{\sqrt m+\sqrt{m-1}}
   >\frac{1}{4\sqrt m}. \]
Combining these with the hypothesis
$\lambda(G)>\sqrt{am}+C$, we obtain
\[ \lambda(G')\ge \lambda(G)-\frac{1}{4\sqrt m}
   >\sqrt{am}+C-\frac{1}{4\sqrt m}
   >\sqrt{am'}+C
   =\sqrt{\Bigl(1-\frac1r+\varepsilon\Bigr)2m'}+C, \]
which contradicts the assumption on $G$.} 
\end{proof}

The following lemma says that if  $G$ is an $m$-edge graph with spectral radius bounded away from $\sqrt{m}$, then every coordinate of the Perron--Frobenius eigenvector of $G$ is small. 

\begin{lemma}[Flatness lemma] \label{lem:max} 
Let $G$ be a graph with $m$ edges and $\bm{x}=(x_v)_{v\in V(G)}$ be the unit Perron--Frobenius eigenvector of $G$.  
 If $\lambda^2 (G) \ge (1+\delta)m$ where $0<\delta \le 0.79$, then 
\[ \max\{x_v : v\in V(G)\} < {\delta^{-4}m^{-1/4} }. \] 
\end{lemma}
\begin{proof}
Let $\lambda = \lambda (G)$ and 
$i\in V(G)$ be a vertex such that 
$x_i= \max\{x_v : v\in V(G)\}$. 
Suppose on the contrary that 
$x_i \ge \delta^{-4}m^{-1/4}$. Note that $\lambda \bm{x}= A_G \bm{x}$ gives 
\[ \lambda x_i = \sum_{j\in N(i)} x_j. \] 
Then we have 
\[  \lambda^2 x_i = \sum_{j\in N(i)} \lambda x_j= \sum_{j\in N(i)} 
\sum_{k\in N(j)} x_k = \sum_{k\in V(G)} d_{N(i)}(k) \cdot x_k , \]
where $d_{N(i)}(k)$ is the number of neighbors of $k$ in $N(i)$. Therefore, we get 
\begin{equation}\label{eq-eigen-sq}
\lambda^2 x_i \le \sum_{k\in V(G)} d_G(k)\cdot x_k = \sum_{jk \in E(G)}(x_j +x_k). 
\end{equation} 
For each edge $jk\in E(G)$, we call $jk$ \emph{good} if $\min\{ x_j ,x_k \} > \delta^{-2} m^{-1/4}$. Since 
$\sum_{v\in V} x_v^2 =1$, we get $\#\{ j\in V(G): x_j > \delta^{-2}m^{-1/4}\}  <  \delta^4 m^{1/2}$. Then 
the number of good edges of $G$ is at most $\delta^8m/2$. 
On the other hand, for each non-good edge 
$jk \in E(G)$, we have 
\[  x_j +x_k 
\le x_i + \delta^{-2} m^{-1/4} \le (1+\delta^2)x_i. \]  
Therefore, it follows from (\ref{eq-eigen-sq}) that 
\[  (1+\delta )m x_i \le \lambda^2 x_i \le 
2x_i \cdot \frac{1}{2} \delta^8m + (1+\delta^2) x_i \cdot m=(1+\delta^2 +\delta^8)mx_i . \]
Since $0<\delta \le 0.79$, we have 
$\delta > \delta^2 + \delta^8$, 
which leads to a contradiction. 
\end{proof}

\begin{lemma}
    \label{lem:remove-edges}
    If $G$ is a graph with $m$ edges, and $G'$ is the graph obtained from $G$ by removing at most $\gamma m$ edges, then $\lambda (G')> \lambda (G) - \sqrt{2\gamma m}$ and 
    $\lambda^2(G') > \lambda^2(G) - 4 \sqrt{\gamma} m.$
\end{lemma}

\begin{proof}
    By Rayleigh's formula, we have $\lambda(G') \geq \lambda(G) - \lambda(G \backslash G')$. Thus
    $$\lambda^2(G') \geq \lambda^2(G) - 2 \lambda(G \backslash G') \lambda(G).$$
    Using the bound $\lambda(G) < \sqrt{2m}$ and $\lambda(G \backslash G') < \sqrt{2 \gamma m}$, we get the desired result.
\end{proof}

The following well-known fact (see, e.g., \cite{BFP2008}) will be used in our argument. 

\begin{lemma} \label{lem-square-equal}
  Let $G$ be a bipartite graph with the partition $V(G)=A\sqcup B$, and let $\binom{\bm{x}}{ \bm{y}}$ be an eigenvector of $G$ corresponding to $\lambda (G)$ with $\bm{x}$ indexed by $A$ and $\bm{y}$ indexed by $B$. Then $\lVert \bm{x}\rVert = \lVert \bm{y}\rVert$, i.e., 
  \[ \sum_{i\in A} x_i^2 = \sum_{j\in B} y_j^2. \]  
\end{lemma}

We here provide a simple proof for the sake of completeness.

\begin{proof}
Since $G$ is bipartite, its adjacency matrix has
the form
\[ A_G=\begin{bmatrix} O & M \\ M^{T} & O
\end{bmatrix}, \]
where $M$ is the $A$--$B$ incidence matrix. 
The eigen-equation gives 
\[ M\bm{y}=\lambda(G)\bm{x} \quad
\text{and}\quad M^{T}\bm{x}=\lambda(G)\bm{y}. \] 
 Hence
\[ \lambda(G)\lVert\bm{x}\rVert^2
 =\bm{x}^{T}\bigl(\lambda(G)\bm{x}\bigr)
 =\bm{x}^{T}M\bm{y}
 =\bigl(M^{T}\bm{x}\bigr)^{T}\bm{y}
 =\lambda(G)\lVert\bm{y}\rVert^2. \]
As $G$ has an edge, we have $\lambda(G)>0$, and therefore
$\lVert\bm{x}\rVert=\lVert\bm{y}\rVert$.
\end{proof}

We will also utilize the well-known Motzkin--Straus inequality \cite{MS1965}. 

\begin{lemma}[Motzkin--Straus \cite{MS1965}] 
\label{lem-MS} 
Let $r\ge 2$ be an integer and $G$ be a $K_{r+1}$-free graph on $n$ vertices. 
    If $\bm{x}=(x_i)_{i=1}^n$ is a non-negative real vector with 
$\sum_{i=1}^n x_i =1$, then 
\begin{equation*}
\sum_{ij \in E(G)} x_ix_j \le \frac{1}{2}\left( 1- \frac{1}{r} \right).
\end{equation*} 
The equality holds if and only if the subgraph induced by the vertices corresponding to non-zero entries of $\bm{x}$ is a complete $r$-partite graph such that the sum of the $x_i$'s in each part is the same. 
\end{lemma}

The following lemma will be used in the proof of Theorem \ref{thm-common}. 

\begin{lemma}[Li--Liu--Zhang \cite{LLZ2024-book-qua}] \label{lem-Delta} 
 Let $G$ be a graph with $m$ edges and maximum degree $\Delta (G)$. 
If $\lambda (G)> \sqrt{m}$, then for sufficiently large $m$, the maximum degree $\Delta(G)$ of $G$ satisfies 
\[ \Delta (G)\le m/2 + m^{0.99}. \]
\end{lemma}

\section{Extensions of the Erd\H{o}s--Stone--Simonovits theorem}\label{sec:ESS}

\subsection{Proof of Theorem \ref{thm-edge-spectral-ESS}}  

\label{subsec-thm1-1and1-2}

Before proving Theorem \ref{thm-edge-spectral-ESS}, 
we present a short proof of Theorem \ref{thm-spec-ESS} which is quite different from Nikiforov's proof in \cite{Niki2009cpc}. 
It illustrates the basic ideas in our proof of Theorem \ref{thm-edge-spectral-ESS}.

\begin{proof}[{\bf Short proof of Theorem \ref{thm-spec-ESS}}]
The lower bound can be witnessed by the Tur\'{a}n graph $T_{n,r}$ since $T_{n,r}$ is $F$-free and $\lambda (T_{n,r}) \ge (1-\frac{1}{r})n- \frac{r}{4n}$. 
Assume that $G$ is an $n$-vertex $F$-free graph. 
By Lemma \ref{lem-EFR}, we can remove $o(n^2)$ edges from 
$G$ and get a $K_{r+1}$-free subgraph $G'$. Note that 
\[ \lambda (G\setminus G') <  
 \sqrt{2e(G\setminus G')} = o(n). \]   
 Since $G'$ is $K_{r+1}$-free, by Wilf's theorem, we have $\lambda (G') \le \left(1- \frac{1}{r} \right) n$. 
 Therefore, we obtain 
\[ \lambda (G) \le \lambda (G') + \lambda (G\setminus G') 
\le \left( 1- \frac{1}{r} \right)n + o(n), \]
which completes the proof.  
\end{proof}

Now we prove Theorem \ref{thm-edge-spectral-ESS}. 
In our proof of Theorem \ref{thm-spec-ESS}, we remove $o(n^2)$ edges from $G$ 
to make it $K_{r+1}$-free, and then apply Wilf's bound. 
For Theorem \ref{thm-edge-spectral-ESS}, we need to show that the removal of these  edges changes the spectral radius by $o(m^{1/2})$, which is true in the dense case when $m=\Omega (n^2)$. Our key idea is to make a reduction to the case of dense graphs using Lemma \ref{lem:max}, which 
gives an upper bound on $\max_{v\in V(G)} \{x_v\}$, where $\bm{x}\in \mathbb{R}^n$ is the unit Perron--Frobenius eigenvector of $G$. 
We note that the assumptions of Lemma \ref{lem:max} do not directly imply that $G$ is dense. For example, $G$ can be the disjoint union of a clique $K_n$ with arbitrarily many isolated vertices. We will address this obstacle by a vertex deletion argument.

\begin{proof}[{\bf Proof of Theorem \ref{thm-edge-spectral-ESS}}]
By Lemma \ref{lem-EFR}, for each $\varepsilon > 0$, there exists some $N(\varepsilon) > 0$ such that any $F$-free graph on $n \geq N(\varepsilon)$ vertices can be made $K_{r + 1}$-free by removing at most $\varepsilon^{20} n^2$ edges. Assume that $G$ is an $F$-free graph with $m$ edges. 
It suffices to prove that for any $\varepsilon  \in (0,1/3)$, we have
$$\lambda(G)\le 
\sqrt{\Bigl(1-\frac{1}{r} + \varepsilon \Bigr)2m} + N(\varepsilon).$$ 
Indeed, once this bound is established,  we have
$\lambda^2(G)\le\bigl(1-\tfrac1r+\varepsilon\bigr)2m
+O_{\varepsilon,F}(\sqrt m\,)$ since
$N(\varepsilon)$ depends only on $\varepsilon$ and $F$.
Letting $\varepsilon\to0$ then yields
$\lambda^2(G)\le\bigl(1-\tfrac1r+o(1)\bigr)2m$, as
claimed.

Suppose on the contrary that 
$G$ is an $F$-free graph with $m$ edges and 
\begin{equation} \label{eq-assume}
    \lambda (G)> \sqrt{\Big(1-\frac{1}{r} + \varepsilon\Big)2m} + N(\varepsilon).
\end{equation}
Among such counterexamples, we choose a graph $G$ such that $|V(G)| + |E(G)|$ is minimal. Denote $|V(G)|=n$. 
By the minimality of $G$, it is clear that $G$ has no isolated vertices and $G$ is connected. 
By the minimality of $G$ again, 
removing any edge does not create a counterexample. 
Hence, for any proper subgraph $G' \subseteq G$ with $m'$ edges, we have 
$\lambda (G')\le \sqrt{(1- \frac{1}{r} + \varepsilon)2m'} +N(\varepsilon)$. 
Using Lemma \ref{lem-product-entry}, for any edge
$uv \in E(G)$, we have 
$$ x_ux_v > \frac{1}{8\sqrt{m}}.$$
In particular, every vertex of $G$ carries a positive weight. Since
$r\ge 2$ implies
$\lambda^2(G)>(1-\frac{1}{r}+\varepsilon)2m
\ge (1+2\varepsilon)m$, we may apply
Lemma \ref{lem:max} with $\delta=2\varepsilon$,
where $2\varepsilon<0.79$, to obtain
\[ \max \{x_v: v\in V(G)\}
   \leq (2\varepsilon)^{-4}m^{-1/4}. \]
For each vertex $v$, picking a neighbor $u$ then
gives $x_v> \frac{1}{8\sqrt{m}}\big/ x_u
\ge 2\varepsilon^4 m^{-1/4}$, and hence
$$\min \{ x_v: v\in V(G)\} > 2 \varepsilon^4 m^{-1/4}.$$ 
As $\bm{x} \in \mathbb{R}^n$ is a unit vector, 
it follows that 
\begin{equation}
    \label{eq-dense}
    n \leq 2^{-2}\varepsilon^{-8} m^{1/2}.
\end{equation}
On the other hand, we get from (\ref{eq-assume}) that $n \geq \lambda(G) \geq N(\varepsilon)$. 
By Lemma \ref{lem-EFR}, we can remove at most $\varepsilon^{20} n^2 $ edges from $G$ to get a $K_{r+1}$-free subgraph $G'$ with $m'$ edges. Note that $\varepsilon^{20} n^2 < \varepsilon^4 m$ by (\ref{eq-dense}). 
Now we use the Nikiforov bound (\ref{eq-Niki-2002-cpc}) and we get 
\[  \lambda (G') \le 
\sqrt{\Bigl(1-\frac{1}{r}\Bigr)2m'}. \]
Since $m'\le m$, we conclude that
\begin{equation*}\label{eq:removal}
    \lambda(G) \le \lambda(G') + \lambda(G\setminus G')
    < \sqrt{\Bigl(1-\tfrac{1}{r}\Bigr)2m'}
      + \sqrt{2 \varepsilon^4 m}
    \le \sqrt{\Bigl(1-\tfrac{1}{r}\Bigr)2m}
      + \sqrt{2 \varepsilon^4 m}.
\end{equation*}
It remains to verify that the right-hand side is
strictly smaller than
$\sqrt{(1-\frac{1}{r}+\varepsilon)2m}$, which
contradicts (\ref{eq-assume}). Indeed, a simple calculation gives 
\begin{align*}
  \sqrt{\Bigl(1-\tfrac{1}{r}+\varepsilon\Bigr)2m}
  -\sqrt{\Bigl(1-\tfrac{1}{r}\Bigr)2m}
  &=\frac{2\varepsilon m}
     {\sqrt{(1-\frac{1}{r}+\varepsilon)2m}
      +\sqrt{(1-\frac{1}{r})2m}} \\
  &>\frac{\varepsilon\sqrt{m}}
     {\sqrt{2(1-\frac{1}{r}+\varepsilon)}}
   >\sqrt{2\varepsilon^4 m},
\end{align*}
where the last inequality holds since
$2\varepsilon\sqrt{1-\frac{1}{r}+\varepsilon}<1$
for all $r\ge 2$ and $\varepsilon\in(0,\tfrac{1}{3})$,
using $1-\frac{1}{r}+\varepsilon<\tfrac{4}{3}$. This
completes the proof. 
\end{proof}

\subsection{Proof of Theorem \ref{thm:walk}}  

To prove~Theorem \ref{thm:walk}, we first show a generalization of Lemma \ref{lem:max} for walks. 
We may assume that $\ell \ge 2$ since the $\ell =1$ case was already proved at the beginning of Subsection \ref{subsec-thm1-1and1-2}.  
Recall that an $\ell$-walk is a walk with $\ell$ vertices. 
For two vertices $i, j$ in a graph $G$, let $w_\ell(i)$ denote the number of $\ell$-walks starting from $i$, and let $w_\ell(i, j)$ denote the number of $\ell$-walks from $i$ to $j$.

\begin{lemma} \label{lem:max-walk}
Let $G$ be a graph and $\bm{x}=(x_j)_{j\in V(G)}$ be its unit Perron--Frobenius eigenvector. If $\lambda^{\ell} (G) \ge \left(\frac{1}{2} + \delta\right) \cdot w_{\ell}(G)$ for some $\delta \in (0, 0.1)$, then 
\[ \max\{x_j : j \in V(G)\} < \delta^{-4}\lambda(G)^{-1/2}. \] 
\end{lemma}

\begin{proof}
We denote  $\lambda = \lambda(G)$ and let $i \in V(G)$ be a vertex with $x_i = \max\{x_j : j\in V(G)\}$. 
Assume for the sake of contradiction that 
$x_i \ge \delta^{-4}\lambda^{-1/2}$. We use the identity
$$\lambda^{\ell} x_i = \sum_{k \in V(G)} A_G^{\ell}(i,k) \cdot x_k = 
 \sum_{k \in V(G)} w_{\ell + 1}(i, k) \cdot x_k.$$ 
Note that each $(\ell +1)$-walk starting from the vertex $i$ 
is determined by the $\ell$-walk obtained by removing $i$. So we have
$$\lambda^{\ell} x_i \leq \sum_{j, k \in V(G)} w_\ell(j, k) \cdot x_k = \frac{1}{2} 
\sum_{j, k \in V(G)} w_\ell(j, k)\cdot  (x_j + x_k).$$
Let $B$ be a vertex subset of $G$ defined as below: 
\[ B := \left\{j\in V(G):  x_j > 
\delta^{-2} \lambda^{-1/2} \right\}. \] 
As $\bm{x}$ is unit, i.e., $\sum_{j\in V(G)} x_j^2 = 1$, we have 
\[ \abs{B} < \delta^4 \lambda.\] 
We call an $\ell$-walk $w$ from $j$ to $k$ \emph{good} if both $j$ and $k$ lie in $B$, and \emph{bad} otherwise.
Applying Rayleigh's formula, the number of good $\ell$-walks in $G$ is given by 
$$\sum_{j, k \in B} w_\ell(j, k) = \mathbf{1}_{B}^T A_G^{\ell - 1}  \mathbf{1}_{B} \leq \lambda^{\ell - 1} \mathbf{1}_{B}^T \mathbf{1}_{B} = \lambda^{\ell - 1} |B| \leq \delta^4 \lambda^{\ell}.$$
Thus we have
\begin{equation} \label{upper-bad}
    \frac{1}{2} \sum_{j, k \in B} w_\ell(j, k) \cdot (x_j + x_k) \leq \delta^4 \lambda^{\ell} x_i.
\end{equation}
On the other hand, for any bad $\ell$-walk from $j$ to $k$, we have 
\[ x_j+x_k\le x_i+\delta^{-2}\lambda^{-1/2}\le(1+\delta^2)x_i. \]
Thus, the contribution from bad walks is bounded by 
\begin{equation} \label{upper-not-bad}
    \frac{1}{2}
\sum_{\substack{j, k \in V(G)\\ \{j, k\} \not\subseteq B}} w_\ell(j, k) \cdot (x_j + x_k) \leq \frac{1}{2} w_{\ell}(G) \cdot (1 + \delta^2) x_i \leq \frac{1 + \delta^2}{1 + \delta} \lambda^{\ell} x_i,
\end{equation}
where the last inequality follows by the assumption  $\lambda^{\ell}\ge \left(\frac{1}{2} + \delta\right) \cdot w_{\ell}(G)$. 

Summing the two inequalities (\ref{upper-bad}) and (\ref{upper-not-bad}), we conclude that
$$\frac{1}{2}\sum_{j, k \in V(G)} w_\ell(j, k) \cdot (x_j + x_k) \leq \left(\frac{1 + \delta^2}{1 + \delta} + \delta^4\right) \lambda^{\ell} x_i < \lambda^{\ell} x_i,$$
which leads to a contradiction.
\end{proof}

An additional lemma we need is the following. 

\begin{lemma}
    \label{lem:walks-lower}
    Let $G$ be a graph and $\bm{x}=(x_j)_{j\in V(G)}$ be its unit Perron--Frobenius eigenvector. Then for any vertex $i$ in $G$, we have
    $$w_\ell(i) \geq \frac{\lambda^{\ell - 1}(G) x_i}{\max_{j \in V(G)} \{x_j\}}.$$
\end{lemma}
\begin{proof}
    Let $\lambda=\lambda(G)$ and $\mathbf{1}_i$ be the vector with $1$ at the $i$-th entry and $0$ everywhere else. Then
    $$\mathbf{1}_i^T A_G^{\ell - 1} \bm{x} = \lambda^{\ell - 1}  \mathbf{1}_i^T \bm{x}  = \lambda^{\ell - 1} x_i.$$
    On the other hand, it follows that 
    $$ \mathbf{1}_i^T A_G^{\ell - 1} \bm{x}   
   = \sum_{j \in V(G)}w_\ell(i, j) \cdot x_j \leq 
     \sum_{j \in V(G)}w_\ell(i, j) \cdot \max_{j \in V(G)} \{ x_j \} = w_\ell(i) \cdot \max_{j \in V(G)} \{ x_j \},$$
    as desired.
\end{proof}

We can now prove Theorem \ref{thm:walk}.

\begin{proof}[{\bf Proof of Theorem \ref{thm:walk}}]
By Lemma \ref{lem-EFR}, for each $\varepsilon > 0$, there exists some $N(\varepsilon) > 0$ such that any $F$-free graph on $n \geq N(\varepsilon)$ vertices can be made $K_{r+ 1}$-free by removing at most $\varepsilon^{20}\ell^{-8} n^2/256$ edges. 
Assume that $G$ is an $F$-free graph with $m$ edges.
Since the case $\ell=1$ is precisely
Theorem~\ref{thm-spec-ESS}, we may assume
$\ell\ge 2$ throughout.
It suffices to prove that, for any
$0<\varepsilon<0.01$, we have 
$$\lambda^{\ell}  (G)\le 
\left(1-\frac{1}{r} + \varepsilon\right) \cdot w_{\ell}(G) + N(\varepsilon)^{\ell}.$$
Suppose  on the contrary that there exists a graph $G$ with
\begin{equation} \label{eq-contra}
\lambda^{\ell} (G) > 
\left(1-\frac{1}{r} + \varepsilon\right) \cdot w_{\ell}(G) + N(\varepsilon)^{\ell}.
\end{equation} 
Among such $G$, we choose a graph $G$ such that
$|V(G)| + |E(G)|$ is minimal. Denote $V = V(G)$.
We may assume that $G$ is connected: otherwise,
the connected component of $G$ with the largest
spectral radius has the same spectral radius as
$G$ while containing no more $\ell$-walks than $G$, and is therefore itself a counterexample to
\eqref{eq-contra} with fewer vertices and edges,
contradicting the minimality of $G$. In
particular, $G$ has no isolated vertex, and by
the Perron--Frobenius theorem its unit
eigenvector $\bm{x}$ is strictly positive, so
$x_v>0$ for every $v\in V$. 

Let $e = \{i, j\}$ be
any edge in $G$, and let $G' = G \backslash \{e\}$
be the subgraph of $G$ obtained by removing $e$.  
By the minimality of $G$, we have
$$\lambda^{\ell}(G') \leq \left(1-\frac{1}{r} + \varepsilon\right) \cdot w_{\ell}(G') + N(\varepsilon)^{\ell}.$$
Therefore, we observe that 
$$\lambda^{\ell}(G) - \lambda^{\ell}(G') > \left(1- \frac{1}{r} + \varepsilon\right) \cdot (w_{\ell}(G) - w_{\ell}(G')).$$ 
On the one hand, we recall that 
$ \lambda(G')\ge \lambda(G) - 2 x_i x_j$. Recall the elementary inequality that $a^{\ell}-b^{\ell} 
\le \ell a^{\ell -1}(a-b)$ holds for any numbers $a\ge b>0$. Then 
\[  \lambda^{\ell}(G)- \lambda^{\ell}(G') \le \ell \cdot \lambda^{\ell -1}(G) \cdot 2x_ix_j.  \]
On the other hand, since any $\ell$-walk in $G$ starting with the edge $ji$ in $G$ no longer appears in $G'$, and there are $w_{\ell - 1}(i)$ such walks, we also have
$$w_{\ell}(G) - w_{\ell}(G') \geq w_{\ell - 1}(i).$$  
Therefore, we have
$$\ell  \cdot \lambda^{\ell - 1}(G)\cdot 2x_i x_j \geq 
\left(1- \frac{1}{r} + \varepsilon\right) \cdot w_{\ell - 1}(i).$$
Invoking Lemma \ref{lem:walks-lower}, we have
$$w_{\ell - 1}(i) \geq \frac{\lambda^{\ell - 2}(G) x_i}{\max_{k \in V} \{x_k\}}.$$
So we conclude that
$$x_j \cdot \max_{k \in V} \{x_k\} \geq \frac{1}{4\ell \lambda(G)}.$$
As each vertex $j$ in $G$ is contained in at least one edge, this inequality holds for any vertex $j$.

By Lemma \ref{lem:max-walk}, we have
$\max \{x_k:k \in V\} < \varepsilon^{-4} \lambda(G)^{-1/2}$. 
Then for any $j \in V$, 
$$x_j \geq \frac{\varepsilon^4}{4\ell} \lambda(G)^{-1/2},$$
which implies
$$\abs{V} \leq 16\ell^2 \varepsilon^{-8} \lambda(G).$$
As $|V|\ge\lambda(G)\ge N(\varepsilon)$, by the removal lemma, we can remove at most $\varepsilon^{20}\ell^{-8} \abs{V}^2/256 \leq \varepsilon^4 \ell^{-4}\lambda^2(G)$ edges from $G$ to obtain a $K_{r+1}$-free subgraph $G'$. 
Applying Theorem \ref{thm-Niki-2006} on $G'$ gives 
$$\lambda(G') \leq \left(\Bigl(1-\frac{1}{r}\Bigr) w_{\ell}(G')\right)^{1/\ell} \leq \left(\Bigl(1-\frac{1}{r}\Bigr) w_{\ell}(G)\right)^{1/\ell}.$$
So we conclude that
\[ \lambda(G) \le \lambda(G') + \lambda(G\setminus G') <\left(\Bigl(1-\frac{1}{r}\Bigr) w_{\ell}(G)\right)^{1/\ell} + \sqrt{2\varepsilon^4 \ell^{-4} \lambda^2(G)}. \]  
Simplifying, we obtain
\[  \lambda^{\ell} (G) < \left(1- \sqrt{2} \varepsilon^2\ell^{-2} \right)^{-\ell} 
 \left(1-\frac{1}{r}\right) w_{\ell}(G), \]
which together with (\ref{eq-contra}) gives 
\[  \left(1- \frac{1}{r} \right)  > \left(1- \sqrt{2} \varepsilon^2\ell^{-2} \right)^{\ell} \left(1-\frac{1}{r} + \varepsilon\right) \ge 
\left(1- \sqrt{2}\varepsilon^2\ell^{-1}\right) \left(1-\frac{1}{r} + \varepsilon\right), \]
where the second inequality follows by Bernoulli's inequality. 
So we get $ \sqrt{2}\varepsilon \ell^{-1}(1-\frac{1}{r}+\varepsilon)>1$, which leads to a contradiction since $\varepsilon< 0.01$.  
The proof is complete. 
\end{proof}

\section{The edge-spectral stability theorem}
\label{sec:edge-stability}

Our proof of~Theorem \ref{thm-edge-spectral-stability} for general graphs $F$ reduces to the case when $F$ is a clique. We first prove the clique case when $F$ is a triangle in~Section \ref{sec:F-triangle},  and then the case when $F=K_{r+1}$ with $r\ge 3$  in~Section \ref{sec:F-clique}. Finally, we settle the general case in~Section \ref{sec:F-general}.

\subsection{The case of triangles}\label{sec:F-triangle}

As a starting point for the whole argument, let us prove Theorem \ref{thm-edge-spectral-stability} when $F$ is the triangle $K_3$. 
For the convenience of readers, we extract this case as the following theorem. 

\begin{theorem} \label{thm-stability-triangle}
    For every $\varepsilon\in (0,0.01)$, 
    there exists $\delta = \delta (\varepsilon) >0$ such that if $G$ is a triangle-free graph with $m$ edges and 
    \[ \lambda (G)\ge (1- \delta )\sqrt{m}, \] 
    then there exist disjoint vertex subsets $A,B \subseteq V(G)$ such that $d(G,K_{A,B})\le \varepsilon m$.
\end{theorem}

Recall that the stability of Mantel's theorem says that 
if $G$ is an $n$-vertex triangle-free graph with 
$e(G)\ge n^2/4 -q$, then 
$G$ can be made bipartite by removing at most $q$ edges; see \cite{Fur2015}. Next, we show a spectral version of this statement.  

\begin{lemma} \label{lem-bipartite}
For every $\varepsilon >0$, 
if $G$ is an $m$-edge triangle-free graph with
\[ \lambda (G) \ge 
(1- \varepsilon /2 ) \sqrt{m}, \] 
then $G$ can be made bipartite by removing at most 
$\varepsilon m$ edges. 
\end{lemma}

\begin{proof}  
Let $u\in V(G)$ be a vertex such that 
$x_u= \max\{x_v: v\in V(G)\}$. 
We denote $U:=N(u)$ and $W:=V(G)\setminus (N(u)\cup \{u\})$.  
Then 
\[  \lambda^2 x_u =\sum_{j\in N(u)} \sum_{k\in N(j)} x_k 
= d(u)x_u + \sum_{v\in U} d_U(v)x_v + 
\sum_{w\in W}d_U(w) x_w. \]
Since $G$ is triangle-free, we have $d_U(v) = 0$ 
for every $v \in U$, and $x_w \le x_u$ for every $w\in W$ by the choice of $u$. Moreover, the assumption gives $\lambda^2 \ge (1-\varepsilon/2)^2 m > (1-\varepsilon)m$. 
Thus, 
$$(1- \varepsilon )mx_u < \lambda^2 x_u \le d(u)x_u + \sum_{w\in W}d_U(w) x_u = (d(u) + e(U, W)) x_u.$$
It follows that 
$$e(W) = 
m- d(u) - e(U,W) <  \varepsilon m.$$ 
After removing all edges within 
$W$, since there are at most $\varepsilon m$ edges in $W$, we obtain the desired bipartite graph with two vertex parts $\{u\}\cup W$ and $U$.  
\end{proof}

We note that this gives the following concise vertex-spectral stability theorem.

\begin{cor}  \label{cor-bipartite-vertex-spectral}
For any $0\le \varepsilon \le 1$ and $n\ge 4/\varepsilon$, if $G$ is an $n$-vertex triangle-free graph with 
\[   \lambda (G) \ge 
(1- \varepsilon/2 )\cdot  \lambda (T_{n,2}), \]  then $G$ can be made bipartite by removing at most $\varepsilon n^2$ edges. 
\end{cor}

\begin{proof}
    The Rayleigh formula implies $\lambda (G)\ge \frac{2m}{n}$, which together with the assumption gives  
    \[ \lambda^2(G)\ge \frac{2m}{n} \lambda (G) \ge 
    \frac{2m}{n} \left(1- \frac{\varepsilon}{2}\right) 
    \lambda (T_{n,2}).\]  
    Note that $\lambda (T_{n,2}) \ge \delta (T_{n,2}) =\lfloor \frac{n}{2}\rfloor \ge \frac{n-1}{2}$. 
    To apply Lemma \ref{lem-bipartite}, it suffices to show that 
    \[ \frac{2m}{n}\left(1- \frac{\varepsilon}{2}\right) \frac{n-1}{2} \ge 
    \left(1-\frac{\varepsilon}{2}\right)^2 m, \] 
which is equivalent to showing that 
\[ \left(1- \frac{\varepsilon}{2}\right) {(n-1)} 
\ge 
\left(1- \varepsilon+ \frac{\varepsilon^2}{4} \right)n. \] 
    The inequality holds by $0\le \varepsilon \le 1$ and 
    $n\ge 4/\varepsilon$, and then  
    $(\frac{\varepsilon}{2} -\frac{\varepsilon^2}{4}) n \ge \frac{1}{4}\varepsilon n > 1- \frac{\varepsilon}{2}$. 
\end{proof}

In what follows, we give another proof of Corollary \ref{cor-bipartite-vertex-spectral}. 

\begin{proof}[Second proof]
Since $G$ is a triangle-free graph with $m$ edges, we know that $\lambda (G)\le \sqrt{m}$. 
    First of all, we claim that $e(G)\ge \lfloor n^2/4 \rfloor  -\varepsilon n^2$. 
    Otherwise, if $e(G)< \lfloor n^2/4 \rfloor  -\varepsilon n^2$, then 
    \[ \lambda (G)\le \sqrt{m} 
    < \sqrt{\big\lfloor {n^2}/{4} \big\rfloor  -\varepsilon n^2} \le \left(1- \tfrac{\varepsilon}{2}\right)\lambda (T_{n,2}), \]  
    where the last inequality holds since $\lambda (T_{n,2}) = \sqrt{\lfloor n^2/4\rfloor}$ and  $\varepsilon \ge 0$. This contradicts the assumption. Thus, we conclude that $e(G)\ge \lfloor n^2/4\rfloor - \varepsilon n^2$. 

    Let $A$ be an independent set of maximum size in $G$. 
Since $G$ is triangle-free, $N(v)$ is an independent set for every $v\in V(G)$, which implies $\Delta(G) \le |A|$. 
Setting $B:=V(G)\setminus A$, we get 
\[ \sum_{v\in B} d(v) \le |B| \Delta (G) \le |B| |A| 
\le \left\lfloor  {n^2}/{4} \right\rfloor.  \]  
On the other hand, it follows that 
\[  \sum_{v\in B} d(v) =e(A,B) + 2e(B) = e(G) + e(B).  \]
Since $e(G)\ge \lfloor {n^2}/{4}\rfloor - \varepsilon n^2$, we get $e(B)\le \varepsilon n^2$. 
Removing all edges within $B$, we can obtain a  bipartite subgraph of $G$ on the vertex parts $A$ and $B$, as needed.   
\end{proof}

To complete the proof of~Theorem \ref{thm-stability-triangle}, 
we show in the rest of this subsection that the resulting bipartite subgraph obtained in Lemma \ref{lem-bipartite} 
is close to a complete bipartite graph. 
Throughout, we write $\bm{x}$ for the unit Perron--Frobenius eigenvector of $G$.

\begin{lemma}
\label{lem:close-to-Kab}
For every $\varepsilon \in (0,0.01)$, 
let $\delta = \varepsilon^4/100$. 
If $G$ is a bipartite graph with 
$m$ edges and $\lambda (G) \ge 
(1- \delta ) \sqrt{m}$, then there exist disjoint vertex subsets $U, V\subseteq V(G)$ such that 
\[ d(G, K_{U, V}) \leq \varepsilon m. \]  
\end{lemma}

\begin{proof}
 Let $V(G)=A\sqcup B$ be a bipartition of $G$. Without loss of generality, we may assume that $A = \{1,2,\ldots ,a\}$ and $B = \{1,2,\ldots ,b\}$, and the unit Perron--Frobenius eigenvector of $G$ takes values $x_1 \geq \dots \geq x_a$ on $A$ and $y_1 \geq \dots \geq y_b$ on $B$.   
 It follows from Lemma \ref{lem-square-equal} that
    $$\sum_{i \in A} x_i^2 = \sum_{j \in B} y_j^2 = \frac{1}{2}.$$
    Therefore, we have
    $$\sum_{i \in A} \sum_{j \in B} x_i^2y_j^2 = \frac{1}{4}.$$
    On the other hand, the assumption gives 
    $$\sum_{ij \in E(G)} 2x_i y_j = \lambda(G) \geq (1 - \delta) \sqrt{m}.$$
    Combining the two inequalities, we obtain
    \begin{equation}
    \label{eq:Motzkin-Straus-Exploit-I}
    \sum_{ij \notin E(G)} x_i^2 y_j^2 + \sum_{ij \in E(G)} \left(x_i y_j - \frac{1}{2\sqrt{m}}\right)^2 \leq \frac{\delta}{2},
    \end{equation}
    where $\sum_{ij \notin E(G)}$ goes over each pair $(i, j) \in (A \times B )\backslash E(G)$ exactly once. 

    \medskip 
    Let $r, s$ be the smallest integers such that
   $\sum_{i = 1}^r x_i^2 \geq \varepsilon^2 
    ~\text{and}~  
    \sum_{j = 1}^s y_j^2 \geq \varepsilon^2$, respectively.   

\begin{claim}\label{cl:1}
We have $\frac{1}{3\sqrt{m}} < x_r y_s < \frac{1}{\sqrt{m}}$.
\end{claim}

    \begin{poc}
    It is easy to verify that $(x-\frac{1}{2\sqrt{m}})^2 \ge \frac{1}{4}x^2$ holds for every $x\in (-\infty, \frac{1}{3\sqrt{m}})\cup  (\frac{1}{\sqrt{m}}, +\infty)$. 
    If $x_r y_s \ge \frac{1}{\sqrt{m}}$, then 
    for any $i \le r$ and $j \le s$, we have $x_i y_j \geq \frac{1}{\sqrt{m}}$. This implies
    $$\min\left\{ x_i^2 y_j^2, \left(x_i y_j- \frac{1}{2\sqrt{m}} \right)^2 \right\} \geq \frac{1}{4} x_i^2 y_j^2.$$
    Summing over $i\le r$ and $j\le s$, we have
    $$\sum_{i=1}^r \sum_{j=1}^s \min\left\{ x_i^2 y_j^2, \left(x_i y_j - \frac{1}{2\sqrt{m}}\right)^2 \right\} \geq \frac{1}{4} \left(\sum_{i =1}^r x_i^2 \right) \left( \sum_{j=1}^s y_j^2 \right) \geq \frac{1}{4} \varepsilon^4$$
    which contradicts (\ref{eq:Motzkin-Straus-Exploit-I}) 
    since $\delta = \varepsilon^4/100$ and $\frac{1}{4}\varepsilon^4 > \frac{1}{2} \delta$. 
   
If $x_r y_s \leq \frac{1}{3\sqrt{m}}$, then
$x_iy_j \le x_ry_s\le \frac{1}{3\sqrt{m}}$ for all $i \ge r$
and $j \ge s$, and hence
$\min\{ x_i^2 y_j^2, (x_i y_j- \frac{1}{2\sqrt{m}})^2\}
\geq \frac{1}{4} x_i^2 y_j^2$ for all such pairs.
By the minimality of $r$ and $s$, we have
$\sum_{i\ge r} x_i^2 = \frac{1}{2} - \sum_{i<r} x_i^2
> \frac{1}{2} - \varepsilon^2$, and similarly
$\sum_{j\ge s} y_j^2 > \frac{1}{2} - \varepsilon^2$.
Summing over $i\ge r$ and $j\ge s$ therefore yields a
contribution of at least
$\frac{1}{4}(\frac{1}{2} - \varepsilon^2)^2 > \frac{\delta}{2}$
to the left-hand side of
(\ref{eq:Motzkin-Straus-Exploit-I}), a contradiction. 
  \end{poc}

     Let $U$ and $V$ be sets of vertices of $G$ defined as  
    \[ 
    U := \left\{u\in A: x_u \geq \frac{x_r}{ 4} \right\} 
    \quad~\text{and}~\quad 
    V := \left\{v \in B: y_v \geq \frac{y_s}{4} \right\}. \]
    In the rest of the proof, we use (\ref{eq:Motzkin-Straus-Exploit-I}) to show that $G$ has small edit distance with $K_{U, V}$. 
    
    First of all, we show that there are few  vertex pairs in $U \times V$ that are non-edges in $G$. 

\begin{claim}\label{cl:2}
We have $|E(K_{U, V}) \backslash E(G)| \leq \frac{1}{3} \varepsilon m $.
 \end{claim}  
  
    \begin{poc}
    We need to show that $G$ misses at most $\frac{1}{3} \varepsilon m$ edges between $U$ 
    and $V$.  By~Claim \ref{cl:1}, for any $(i, j) \in (U\times V) \backslash E(G)$, we have 
    $$x_i y_j \geq \frac{1}{16} x_r y_s  \geq  \frac{1}{48 \sqrt{m}}.$$ 
    Using (\ref{eq:Motzkin-Straus-Exploit-I}), we obtain
        $$|E(K_{U, V}) \backslash E(G)| \leq 
        (48\sqrt{m})^2 \cdot \sum_{ij \notin E(G)} x_i^2 y_j^2
        \leq 
        (48 \sqrt{m})^{2} \cdot \frac{\delta}{2} 
        \leq 1152 \delta m \leq \frac{1}{3} \varepsilon m ,$$
        as desired.
    \end{poc}

    Moreover, we prove that there are few edges in $G[\bar{U},B]$ and 
    $G[A,\bar{V}]$. 

\begin{claim}\label{cl:3}
    We have $e(E(G) \cap (\bar{U} \times B)) \leq \frac{1}{3} \varepsilon m $, and symmetrically $e(E(G) \cap (A \times \bar{V})) \leq \frac{1}{3} \varepsilon m $.
\end{claim} 
    
    \begin{poc}
        Let $E_1$ be the set of edges $ij$ in $E(G) \cap (\Bar{U} \times B)$ with $j \geq s$, and let $E_2$ be the set of edges $ij$ in $E(G) \cap (\Bar{U} \times B)$ with $j < s$. It suffices to show that $\max\{\abs{E_1}, \abs{E_2}\} \leq 
        \frac{1}{6} \varepsilon m $.

        For any $ij \in E_1$, by Claim \ref{cl:1}, we have $x_i y_j \leq \frac{1}{4} x_r y_s  \leq \frac{1}{4\sqrt{m}}$. Then together with (\ref{eq:Motzkin-Straus-Exploit-I}), we have 
        $$\abs{E_1} \cdot \frac{1}{16 m} \leq \sum_{ij \in E_1} \left(x_i y_j - \frac{1}{2\sqrt{m}} \right)^2 \leq \frac{1}{2}\delta , $$
        which implies $\abs{E_1} \leq 8 \delta m < \frac{1}{6} \varepsilon m $, as needed. 

        To bound $|E_2|$, we use the Cauchy--Schwarz inequality to obtain
        \begin{align}
       \notag \sum_{ij \in E(G) \backslash E_2} 2x_i y_j  & \leq 2 \sqrt{\sum_{ij \in E(G) \backslash E_2} x_i^2 y_j^2} \cdot \sqrt{ |E(G) \backslash E_2| } \\ 
        & \leq 2 \sqrt{\sum_{i \in A, j \in B} x_i^2 y_j^2} \cdot \sqrt{|E(G) \backslash E_2|} \notag \\
        &= \sqrt{m - |E_2| }. \label{eq-notin-E2}
        \end{align}
        By the definition of $s$, we have $\sum_{j < s} y_j^2 < \varepsilon^2$. Observe that 
        $$\sum_{ij \in E_2} x_i^2 y_j^2 \leq \left(\sum_{i \in A} x_i^2\right) \left( \sum_{j < s} y_j^2 \right) < \frac{\varepsilon^2}{2}.$$
        Applying the Cauchy--Schwarz inequality, we obtain
        \begin{align} \label{eq-in-E2}
        \sum_{ij \in E_2} 2x_i y_j \leq 2 \sqrt{\sum_{ij \in E_2} x_i^2 y_j^2} \cdot \sqrt{\abs{E_2}} \leq \sqrt{2 \varepsilon^2 \abs{E_2}}.
        \end{align}
        Combining (\ref{eq-notin-E2}) and (\ref{eq-in-E2}), we conclude that 
        $$ 
        (1 - \delta) \sqrt{m} \le 
        \lambda(G) = \sum_{ij \in E(G)} 2x_i y_j 
        \le \sqrt{m - \abs{E_2}} + \sqrt{2 \varepsilon^2 \abs{E_2}}.$$
        By the Cauchy--Schwarz inequality again, we have
        $$\sqrt{m - \abs{E_2}} + \sqrt{2 \varepsilon^2 \abs{E_2}} \leq \sqrt{(1 + 4\varepsilon^2) \left(m - \abs{E_2} + \frac{1}{2}\abs{E_2}\right)}.$$
        Therefore, we obtain
        $$ (1 - \delta)^2 m\le 
        (1 + 4\varepsilon^2) \left(m - \abs{E_2} + \frac{1}{2}\abs{E_2} \right), $$
        which implies that $\abs{E_2} \leq \frac{1}{6} \varepsilon m $ by our choice of $\delta = \varepsilon^4 / 100$ and $\varepsilon \in (0, 0.01)$.
    \end{poc}
    Combining~Claims \ref{cl:2} and~\ref{cl:3}, we conclude that 
    $$d(G, K_{U, V}) \leq e(K_{U, V} \backslash G) + e(E(G) \cap (\bar{U} \times B)) + e(E(G) \cap (A \times \bar{V})) \leq \varepsilon m, $$ 
    as desired.
\end{proof}

We can now prove Theorem \ref{thm-stability-triangle}.

\begin{proof}[{\bf Proof of Theorem \ref{thm-stability-triangle}}]
We take $\delta = \frac{1}{4}\varepsilon^{36}$. 
Assume that $G$ is a triangle-free graph with $m$ edges and $\lambda (G)\ge (1- \frac{1}{4}\varepsilon^{36})\sqrt{m}$. 
Using Lemma \ref{lem-bipartite}, we obtain a bipartite subgraph $G'\subseteq G$ such that $G'$ has $m'\ge m - \frac{1}{2}\varepsilon^{36} m $ edges. 
Note that Rayleigh's formula gives 
$$\lambda (G)\le \lambda (G') + \lambda (G\setminus G') \le \lambda (G') + \sqrt{2(m-m')}\le \lambda (G') + \varepsilon^6\sqrt{m}.$$ 
Then $\lambda (G')\ge (1-\delta)\sqrt{m} - \varepsilon^6 \sqrt{m} \ge (1-2\varepsilon^6) \sqrt{m'}> (1-\frac{\varepsilon^4}{1600})\sqrt{m'}$. Now, we can apply~Lemma \ref{lem:close-to-Kab} to $G'$ to obtain two disjoint vertex sets $U,V \subseteq V(G')$ such that $d(G',K_{U,V})\le \frac{\varepsilon}{2}m$. Thus, we conclude that $d(G,K_{U,V})\le d(G,G') + d(G',K_{U,V})\le \frac{1}{2}\varepsilon^{36} m + \frac{1}{2}\varepsilon m< \varepsilon m$, as desired. 
\end{proof}

\subsection{The case of cliques}\label{sec:F-clique}

In this subsection, we prove Theorem \ref{thm-edge-spectral-stability} for larger forbidden cliques $F = K_{r + 1}$ with $r\ge 3$. 
We rephrase this case as the following theorem. 

\begin{theorem} 
\label{thm-stability-clique-4}
    For every $r\ge 3$ and $\varepsilon>0$, there exist $m_0$ and $\delta >0$ such that if $G$ is a $K_{r+1}$-free graph with $m\ge m_0$ edges and 
    \[ \lambda^2(G)\ge 
    \left( 1-\frac{1}{r}-\delta \right) 2m, \] 
    then there exists a vertex set $C\subseteq V(G)$ such that $d(G,T_{C,r})\le \varepsilon m$. 
\end{theorem}

We provide two different proofs of Theorem \ref{thm-stability-clique-4}. 
The first proof is based on the Motzkin--Straus inequality (Lemma \ref{lem-MS}) and reduces the edge-spectral stability in Theorem \ref{thm-stability-clique-4} to the vertex-spectral stability in Theorem \ref{thm-niki-vertex-stability}. 
The second alternative proof applies a recent result involving the stability of the Kruskal--Katona type result for hypergraphs. We present the first proof here and postpone the second proof to the Appendix \ref{sec-A}. 

First of all, we exploit the slackness in Nikiforov's application of the Motzkin--Straus inequality. Using this slackness, we show that $G$ contains an induced subgraph $G'$ such that $d(G,G')=o(m)$ 
and $\lambda (G')= (1-\frac{1}{r}- o(1))|G'|$. This puts us in a position to use Theorem \ref{thm-niki-vertex-stability}.

\begin{lemma}
\label{lem:nikiforov-stability-clique}  
For every $\varepsilon > 0$ and $r \geq 3$, 
there exist $\delta>0$ and  $m_0$ such that the following holds. If $G$ is a $K_{r + 1}$-free graph with $m\ge m_0$ edges and $\lambda^2 (G)\ge (1- \frac{1}{r} - \delta) 2m$, then $G$ has an induced subgraph $G'$ such that $\abs{V(G')} \leq (1 + \varepsilon) \sqrt{2m / (1 - 1/r)}$ and $d(G, G') \leq \varepsilon m$.
\end{lemma}

\begin{proof}
    In the following argument, we can take $\delta = (0.1\varepsilon)^8$.  
    For simplicity, we write $V$ and $E$ for the set of vertices and the set of edges of $G$, respectively, and 
    let $\bm{x}=(x_i)_{i\in V}$ be the unit Perron--Frobenius eigenvector of $G$. 
    Setting $\alpha :=(\frac{r-1}{r}\cdot \frac{1}{2m})^{1/4}$, and expanding the
    square,
    \[ \sum_{ij\in E}\big(x_ix_j-\alpha^2\big)^2
       =\sum_{ij\in E}x_i^2x_j^2
        -2\alpha^2\!\sum_{ij\in E}x_ix_j+\alpha^4 m. \]
    The first term is at most $\frac{r-1}{2r}$ by Lemma \ref{lem-MS},
    and the last term equals $\frac{r-1}{2r}$. Since the coefficient
    $-2\alpha^2$ is negative, the Rayleigh bound
    $\sum_{ij\in E}x_ix_j=\frac12\lambda(G)
     \ge\frac12\sqrt{(\frac{r-1}{r}-\delta)2m}$ gives
    \[ -2\alpha^2\!\sum_{ij\in E}x_ix_j
       \le-\alpha^2\sqrt{\Big(\frac{r-1}{r}-\delta\Big)2m}
       =-\sqrt{\frac{r-1}{r}\Big(\frac{r-1}{r}-\delta\Big)}. \]
    Combining the three bounds, we obtain 
    \begin{equation} \label{eq-expect}
    \sum_{ij \in E} \left(x_i x_j - \alpha^2\right)^2 \leq \frac{r - 1}{r} - \sqrt{\frac{r-1}{r} \cdot \Big(\frac{r - 1}{r}-\delta\Big)} < \delta,
    \end{equation}
    where the last inequality is $a-\sqrt{a(a-\delta)}<\delta$ with
    $a=\frac{r-1}{r}$, which follows from
    $1-\sqrt{1-\delta/a}<\delta/a$. 
    We say that an edge $ij \in E$ is \emph{bad} if $\abs{x_i x_j - \alpha^2} > \delta^{1/4} \alpha^2$, and \emph{good} otherwise. Let $G_b$ be the subgraph of bad edges, and $G_g$ be the subgraph of good edges. 
    For simplicity, we write $E_b$ and $E_g$ for the edge sets of $G_b$ and $G_g$, respectively. 
    From (\ref{eq-expect}), we have 
    \begin{equation} \label{upper-Eb}
         |E_b| \le \frac{\delta}{\delta^{1/2}\alpha^4} \leq 4\delta^{1/2} m.
    \end{equation}
    Using this fact, we now show that almost all vertices $i$ satisfy $x_i \approx \alpha$.

    For each positive integer $t \geq 1$, let 
    $$\Gamma_t:=\left\{i \in V: |x_i - \alpha| \geq 2\delta^{1/4} \alpha t \right\}.$$ 
    Observe that $\Gamma_t \subset \Gamma_{t - 1}$. We proceed via the following claims.

    The first one follows by the definition of $G_g$.
    
    \begin{claim}\label{cl:obs1}
    The graph $G_g[\Gamma_1]$ is a bipartite graph on the vertex parts 
   \[ \Gamma_1^+ := \{i \in V: x_i \geq (1 + 2\delta^{1/4}) \alpha\} \]
    and 
    \[ \Gamma_1^- := \{i \in V: x_i \leq (1 - 2\delta^{1/4}) \alpha\}. \]
    \end{claim}

    \begin{claim}\label{cl:obs2}
    For each $t \leq \frac{1}{2}\delta^{-1/8} $, all the edges between $\Gamma_t$ and $V \backslash \Gamma_{t - 1}$ are bad. 
    \end{claim}
    
    \begin{poc}
    Indeed, if $i \in \Gamma_t$ and $j \in V \backslash \Gamma_{t - 1}$, then 
    \[ 
    |x_i -\alpha|\ge 2\delta^{1/4} \alpha t  
    \quad \text{and}  \quad 
    |x_j - \alpha| < 2\delta^{1/4}\alpha (t-1). \]  Therefore, we obtain  
    \begin{align*}
    \abs{x_i x_j - \alpha^2} &= \abs{(x_i - \alpha)x_j + (x_j - \alpha)\alpha} \geq x_j\abs{x_i - \alpha} - \alpha 
    \abs{x_j - \alpha} \\
    &\geq (1 - 2(t - 1)\delta^{1/4}) \cdot 2t\delta^{1/4}\alpha^2  -
    2(t - 1)\delta^{1/4} \alpha^2 \\
    &> 2\delta^{1/4} \alpha^2 
    - 4t^2 \delta^{1/2} \alpha^2\geq \delta^{1/4} \alpha^2, 
    \end{align*}
    where the last inequality holds since $t\le \frac{1}{2} \delta^{-1/8}$. 
    So the edge $ij$ is bad.
    \end{poc}

    We now consider the quantities 
    $$s_t := \sum_{i \in \Gamma_t} x_i^2.$$
    Our goal is to find an integer $\ell$ such that 
    $s_{\ell} \le O(\delta^{1/4})$, and then we will show that there are at most $\varepsilon m$ edges incident to the vertices of $\Gamma_{\ell}$ and that deleting all vertices of $\Gamma_{\ell}$ yields the desired subgraph $G'$.

    For every $t \le \frac{1}{2} \delta^{-1/8} $, as $\bm{x}$ is the Perron--Frobenius eigenvector, we have
    $$\lambda(G) s_t = \sum_{i \in \Gamma_t} \sum_{j \in N(i)} x_i x_j 
    =\sum_{i \in \Gamma_t} \sum_{\substack{j \in N(i)\\ j\in \Gamma_{t-1}}} x_ix_j + 
    \sum_{i \in \Gamma_t} \sum_{\substack{j \in N(i)\\ j\in V\setminus \Gamma_{t-1}}} x_ix_j.$$
    Note that $\Gamma_t \subset \Gamma_{t-1}$ for each $t\ge 1$. 
    By~Claim \ref{cl:obs2}, every edge summed over on the right hand side is either in $G_g[\Gamma_{t - 1}]$ or bad. Then 
    $$\lambda(G) s_t \leq \sum_{ij \in E_g[\Gamma_{t - 1}]} 2x_i x_j + \sum_{ij \in E_b} x_i x_j.$$
    By~Claim \ref{cl:obs1}, $G_g[\Gamma_{t - 1}]\subseteq G_g[\Gamma_{1}]$ is bipartite, so its spectral radius is at most $\sqrt{m}$. Thus, 
    $$\sum_{ij \in E_g[\Gamma_{t - 1}]} 2x_i x_j \leq \sqrt{m}\sum_{i \in \Gamma_{t - 1}} x_i^2 = \sqrt{m} s_{t - 1}.$$
    On the other hand, recall in (\ref{upper-Eb}) that  
    $|E_b|\le 4\delta^{1/2}m$. Then    
    $$\sum_{ij \in E_b} x_i x_j < \frac{1}{2} \sqrt{2 \abs{E_b}} < 2 \delta^{1/4} \sqrt{m}.$$
    Since $r\ge 3$ and $\lambda(G) \geq \frac{10}{9} \sqrt{m}$, we conclude that 
    $$s_t \leq 0.9\, s_{t - 1} + 2\delta^{1/4}.$$
  Write
    $s^*:=20\delta^{1/4}$ and 
    $\rho_t:=s_t-s^*$.  Subtracting $s^*$ from both sides, we get 
    $$\rho_t \le 0.9\,s_{t-1}+2\delta^{1/4}-s^*
      = 0.9\,\rho_{t-1}+\big(2\delta^{1/4}-0.1\,s^*\big)
      = 0.9\,\rho_{t-1}.$$
Combining with the initial condition  $\rho_1 \le s_1 \le 1$, we obtain 
$$\rho_t\le 0.9^{t-1}\rho_1 \le 0.9^{t-1}.$$
 Then for every $t \leq \frac{1}{2} \delta^{-1/8} $, we have 
    $$s_t = s^*+\rho_t \le 20 \delta^{1/4} + 0.9^{\,t - 1}.$$ 
    Taking $\ell := \floor{ \frac{1}{2}\delta^{-1/8}}$, since  $\delta$ is sufficiently small, we get $0.9^{\ell -1} < 10 \delta^{1/4}$. Thus, 
    \begin{equation} \label{eq-sl}
    s_{\ell} \leq 30 \delta^{1/4}.
    \end{equation} 
    Let $V':=V(G)\setminus \Gamma_{\ell}$ and  $G':=G[V']$. 
    In what follows, we verify that $G'$ has the desired properties. 
    Firstly, by the definition of $V'$, we know that for every $i \in V'$,  
    $$x_i \geq (1 - 2\delta^{1/4} \ell) \alpha \geq (1 - 2\delta^{1/8}) \alpha.$$
    Recall that $\bm{x}$ is a unit vector and 
    $\alpha = (\frac{r-1}{r}\cdot \frac{1}{2m})^{1/4}$. 
    Thus, we get 
    $$
    |V'|   \leq \frac{1}{(1 - 2\delta^{1/8})^2 \alpha^2} 
    \sum_{i\in V'} x_i^2 \leq \frac{1}{(1 - 2\delta^{1/8})^2 \alpha^2} \leq (1 + \varepsilon) \cdot \sqrt{\frac{2rm}{r - 1}}.
$$    
Secondly, let $R$ be the set of edges with at least one vertex in $\Gamma_{\ell}$. Recall by~\eqref{eq-expect} that
    $$\sum_{ij \in R} \left(x_i x_j - \alpha^2\right)^2 \leq \sum_{ij \in E} \left(x_i x_j - \alpha^2\right)^2 \leq \delta.$$
    By the Cauchy--Schwarz inequality, we have
    $$ |R| \alpha^2 - \sum_{ij\in R} x_ix_j = 
    \sum_{ij \in R} (\alpha^2 - x_i x_j) \leq \sqrt{\delta \abs{R}}.$$
    On the other hand, we have
    $$\sum_{ij \in R} x_i x_j \leq \sum_{i\in \Gamma_{\ell}} x_i \sum_{j\in N(i)} x_j = 
    \lambda(G) \sum_{i \in \Gamma_{\ell}} x_i^2 < 30 \sqrt{2m} \delta^{1/4},$$
    where the last inequality holds by (\ref{eq-sl}). 
    So we conclude that
    $$\frac{1}{2\sqrt{m}}\abs{R} \leq \alpha^2 \abs{R} < 30 \sqrt{2m} \delta^{1/4} + \sqrt{\delta \abs{R}}.$$
    Solving this inequality, we obtain  $\abs{R} \leq 200\delta^{1/4} m < \varepsilon m$, which completes the proof.
\end{proof}

Now, we are ready to prove Theorem \ref{thm-stability-clique-4} 
by combining Lemma \ref{lem:nikiforov-stability-clique} with Theorem \ref{thm-niki-vertex-stability}. 

\begin{proof}[{\bf Proof of Theorem \ref{thm-stability-clique-4}}] 
For any $\varepsilon>0$, let $\varepsilon'>0$ be sufficiently small so that $\varepsilon'<\varepsilon/2$ and $ 4\sqrt{\varepsilon'} <\delta_{\ref{thm-niki-vertex-stability}}(\varepsilon/8)$. Set $\delta=\delta_{\ref{lem:nikiforov-stability-clique}}(\varepsilon')$. Assume that $G$ is an $m$-edge $K_{r+1}$-free graph with $\lambda^2 (G)\ge {(1- \frac{1}{r} - \delta)2m}$. 
Applying Lemma \ref{lem:nikiforov-stability-clique} to $G$ with $\varepsilon_{\ref{lem:nikiforov-stability-clique}}=\varepsilon'$, we obtain a subgraph $G'$ on vertex set $V'$ with 
\[  |V'|\le (1+\varepsilon')\sqrt{\frac{2m}{1-1/r}} 
\quad \text{and} \quad d(G,G')\le \varepsilon' m. \] 
Then we can apply the vertex-spectral stability (Theorem \ref{thm-niki-vertex-stability}) to the $K_{r+1}$-free graph $G'$. 
Indeed, using Rayleigh's formula, we have 
\[ \lambda (G')\ge \sqrt{\Big(1- \frac{1}{r} - \delta \Big)2m} - \sqrt{2\varepsilon' m} 
\ge \sqrt{\Big(1- \frac{1}{r} - 3\sqrt{\varepsilon'} \Big)2m}. \] 
Writing
$a:=1-\tfrac1r\in[\tfrac23,1)$, the above bounds give 
\begin{align*}
\frac{\lambda(G')}{|V'|}
&\ge \frac{\sqrt{(a-3\sqrt{\varepsilon'})\,2m}}
          {(1+\varepsilon')\sqrt{2m/a}}
 = \frac{1}{1+\varepsilon'}\sqrt{a\,(a-3\sqrt{\varepsilon'})} \\
&\ge (1-\varepsilon')\big(a-3\sqrt{\varepsilon'}\big)
 \ge a-4\sqrt{\varepsilon'}. 
\end{align*}
 Hence, we obtain 
\[ \lambda (G') \ge \left( 1-\frac{1}{r} - 4\sqrt{\varepsilon'} \right) |V'|. \]  
We can choose $m$ sufficiently large such that $|V'|> \lambda (G') \ge \sqrt{m} \ge (n_0)_{\ref{thm-niki-vertex-stability}}$. 
Since $ 4\sqrt{\varepsilon'} <\delta_{\ref{thm-niki-vertex-stability}}(\varepsilon/8)$, applying Theorem \ref{thm-niki-vertex-stability} with $\varepsilon_{\ref{thm-niki-vertex-stability}}=\varepsilon/8$, 
we know that $G'$ differs from the $r$-partite Tur\'{a}n graph 
$T_{V',r}$ in at most 
$\varepsilon|V'|^2/8\le \varepsilon m/2$ edges. 
So it follows that 
$$d(G,T_{V',r}) \le d(G,G') + d(G',T_{V',r}) \le \varepsilon'm+\varepsilon m/2\le \varepsilon m,$$
as desired.
\end{proof}


\subsection{General graphs: Proving Theorem \ref{thm-edge-spectral-stability}}\label{sec:F-general}
In this subsection, we employ our framework in the proof of Theorem \ref{thm-edge-spectral-ESS} to complete the proof of Theorem \ref{thm-edge-spectral-stability} for general graphs $F$. 
First, we handle the case when $\chi (F)=r+1$ with $r \geq 3$. The main step is to find a large subgraph $G'$ retaining most of the edges such that its Perron--Frobenius eigenvector satisfies $x_i'x_j'\ge \delta m^{-1/2}$ for each $ij\in E(G')$. This in conjunction with Lemma \ref{lem:max} implies that $G'$ is dense: it has $O(\sqrt{m})$ vertices. We can then apply Lemma \ref{lem-EFR} to $G'$ to obtain a $K_{r+1}$-free subgraph $G''$ by removing $o(|G'|^2)$ edges. This would complete the proof as $\lambda (G'')=(1-\frac{1}{r} -o(1))2m$ and we reduce the problem to the clique case when $F = K_{r + 1}$.

We now implement this strategy in detail. To find the desired $G'$ mentioned above, we need the following lemma about the stability of the spectral radius after removing few edges. It eliminates some ``bad'' edges of $G$ that have small contribution to $\lambda(G)$.

\begin{lemma}
\label{lem:removing-bad-edges}
Let $F$ be a graph with $\chi (F)=r+1\ge 3$. For every $\varepsilon \in (0, 0.01)$, there exists $m_0$,  depending only on $F$ and $\varepsilon$, such that the following holds. If $G$ is an $F$-free graph with $m\ge m_0$ edges and $\lambda^2 (G)\ge (1- \frac{1}{r} - \varepsilon) 2m$, then $G$ contains a subgraph $G'$ with  $m'\ge (1 - 4\eps) m$ edges such that the unit Perron--Frobenius eigenvector $\bm{x}'$ of $G'$ satisfies that for every  $ij \in E(G')$, 
$$x'_i x'_j \geq \varepsilon m^{-1/2}.$$
\end{lemma}

\begin{proof}
    Let $\bm{x}=(x_i)_{i\in V}$ be the unit Perron--Frobenius eigenvector of $G$. 
    If $e = ij$ is an edge of $G$, then the Rayleigh quotient implies 
    $$\lambda(G \backslash e) \geq \lambda(G) - 2x_i x_j.$$ 
    We start with $G' = G$, and repeatedly delete edges $ij$ in $G'$ with $x'_i x'_j < \varepsilon m^{-1/2}$ from $G'$. Suppose on the contrary that we can repeat this deletion for $\ceil{4\varepsilon m}$ steps. Then the resulting graph $G'$ has $m' \leq (1 - 4\varepsilon)m$ edges, and satisfies
    $$\lambda(G') \geq \lambda(G) - 2\varepsilon m^{-1/2} \cdot \ceil{4\varepsilon m} \geq \lambda(G) - 9\varepsilon^2 m^{1/2}.$$
    Squaring the above inequality, we have 
    $$\lambda^2(G') > \lambda^2(G) - 18\varepsilon^2 m^{1/2} \lambda(G) > 
    \left(1- \frac{1}{r} - \varepsilon \right) 2m - 36 \varepsilon^2 m > 
    \left(1- \frac{1}{r} + \frac{\varepsilon}{4} \right) 2m'.$$
    However, $G'$ is $F$-free. For sufficiently large $m$, this contradicts Theorem \ref{thm-edge-spectral-ESS}.
\end{proof} 

Now, we are ready to prove the edge-spectral stability in Theorem \ref{thm-edge-spectral-stability} when $r\ge 3$. 

\begin{proof}[{\bf Proof of part (b) of Theorem \ref{thm-edge-spectral-stability}}]
    Let $\delta_0$ be the value of $\delta$ for Theorem  \ref{thm-stability-clique-4} with $\varepsilon' = \varepsilon / 3$. Without loss of generality, we may assume that $\delta_0 < \min\{10^{-10},  \varepsilon/10\}$. We take $\delta = \delta_0^3$. Let $G$ be an $F$-free graph with $m\ge m_0$ edges and $\lambda^2 (G)\ge (1- \frac{1}{r} - \delta) 2m$.
    
    We throw away a few edges of $G$ to obtain a large $K_{r+1}$-free subgraph. Applying Lemma \ref{lem:removing-bad-edges} to $G$, we find a subgraph $G'$ of $G$ with  $m'\ge (1 - 4\delta) m$ edges such that the Perron--Frobenius eigenvector $\bm{x}'=(x_i')_{i\in V(G')}$  satisfies $x'_i x'_j \geq \delta m^{-1/2}$ for every edge $ij$ of $G'$. 
    After removing all isolated vertices from $G'$, we let $V'$ be its vertex set. By Lemma \ref{lem:remove-edges}, we have
    $$\lambda^2(G') > \lambda^2(G) - 4\sqrt{4\delta} m > \left(1- \frac{1}{r} - 5\sqrt{\delta} \right) 2m.$$
    Since $r \geq 3$, we have $\lambda^2(G')  \geq 1.3 m$. By Lemma \ref{lem:max}, we have 
    \[  \max\{x'_i:i \in V(G') \} \leq 124 (m')^{-1/4} \leq 125 m^{-1/4}. \] 
     Thus, we have 
     \[ \min\{x'_i :i \in V(G')\} \ge 125^{-1}\delta m^{-1/4} . \] 
     Note that $1=\sum_{i\in V'} x_i'^2 
     \ge |V'| 125^{-2}\delta^2m^{-1/2}$, 
     which implies $\abs{V'} \le 125^2 \delta^{-2} m^{1/2}$.

    Applying Lemma \ref{lem-EFR}, 
    there exists an $m_0$ such that 
    for $m \ge m_0$, we can remove at most $125^{-4}\delta^5 \cdot |V'|^2 \le \delta m$ edges from $G'$ to form a new subgraph $G''$ that is $K_{r + 1}$-free. By Lemma \ref{lem:remove-edges}, we have
    $$\lambda^2(G'') > \lambda^2(G) - 4\sqrt{5\delta} m > \left(1- \frac{1}{r} - 5 \sqrt{\delta} \right) 2m > \left(1- \frac{1}{r} - \delta_0 \right) 2m.$$
    Applying Theorem \ref{thm-stability-clique-4}, there exists a vertex set $C \subseteq V(G'')$ such that
    $$d(G'', T_{C, r}) \leq \varepsilon m / 3.$$
    Thus we obtain
    \begin{align*}
    d(G, T_{C, r}) &\leq  d(G, G')+ d(G',G'')  +d(G'', T_{C, r}) \\
    &\leq 4 \delta m + \delta m +  \varepsilon m / 3  \\
    & \leq \varepsilon m , 
    \end{align*}
    as desired.
\end{proof}

For the case $r = 2$,  we cannot apply Lemma \ref{lem:max} to conclude that all the coordinates $(x_v)_{v\in V}$ are bounded by $O(m^{-1/4})$. To overcome this obstacle, we use instead the following lemma.

\begin{lemma} \label{lem:condition-max}
For every $\varepsilon > 0$, there exists $\delta > 0$ such that the following holds. Let $G$ be a graph with at most $m$ edges, and let $\bm{x}$ be the unit Perron--Frobenius eigenvector of $G$. If $\lambda^2 (G) \ge (1 - \delta)m$ and 
\begin{equation} \label{eq-xv-large}
\max\{x_i : i\in V(G)\} > {\delta}^{-1} m^{-1/4}, 
\end{equation}
then there are disjoint vertex subsets $U, V$ of $G$ such that $d(G, K_{U, V}) \leq \varepsilon m$.
\end{lemma}

\begin{proof}
Under the condition (\ref{eq-xv-large}), we first show that $G$ contains a large bipartite subgraph $G'$ with large spectral radius. Then we can invoke Lemma \ref{lem:close-to-Kab} to complete the proof. 
We set $\delta := (\delta_0/5)^8$,  
where $\delta_0=(\varepsilon/2)^4/100$ is the constant from Lemma \ref{lem:close-to-Kab}. 
Let $\lambda = \lambda (G)$ and 
$x_i= \max\{x_v : v\in V(G)\}$. 
Since $\lambda x_i = \sum_{j\in N(i)} x_j$, we have 
\begin{equation}\label{eq-eigen}
  \lambda^2 x_i = \sum_{j\in N(i)} \lambda x_j= \sum_{j\in N(i)} 
\sum_{k\in N(j)} x_k \le \sum_{jk \in E(G)}(x_j +x_k). 
\end{equation}
Let 
$$R := \{ j\in V(G): x_j > \delta^{-1/2}m^{-1/4}\}.$$ 
Since 
$\sum_{v\in V} x_v^2 =1$, we see that 
\[ \abs{R} <  \delta m^{1/2}. \] 
We obtain from (\ref{eq-eigen}) that 
$$ (1-\delta)mx_i\le \lambda^2 x_i \leq \sum_{jk \in E(R)}(x_j +x_k) +  \sum_{ jk \in E(\bar{R})}(x_j +x_k) +  \sum_{jk \in E(R, \bar{R})}(x_j +x_k).$$
The first term is at most $\abs{R}^2 x_i \leq \delta^2 m x_i$. By the definition of $R$ and the fact that $x_i > \delta^{-1}m^{-1/4}$ due to \eqref{eq-xv-large}, the second term above is at most $m \cdot 2\delta^{-1/2}m^{-1/4} \leq 2\delta^{1/2} m x_i$. Thus we have
$$\sum_{jk \in E(R, \bar{R})}(x_j +x_k) \geq (1 - 4\delta^{1/2}) m x_i.$$
For each edge $jk \in E(R, \bar{R})$, we have $x_j + x_k \leq x_i + \delta^{-1/2} m^{-1/4} \le 
(1 + \delta^{1/2})x_i$. Then 
\[ e(R,\bar{R})\cdot  (1+ \delta^{1/2})x_i 
\ge (1- 4\delta^{1/2} ) mx_i. \]
We conclude that at least $(1 - 5\delta^{1/2})m$ edges lie between $R$ and $\bar{R}$. Now, we denote $G'=G[R,\bar{R}]$. 
Then $d(G,G')= e(G)- e(G')\le 5\delta^{1/2}m$. Finally, 
Lemma \ref{lem:close-to-Kab} can be invoked on $G'$ to complete the proof. 
Indeed, Rayleigh's formula yields 
\[ \lambda (G') \ge \lambda (G) - 
\lambda (G\setminus G') 
\ge \sqrt{(1-\delta) m} - \sqrt{10\delta^{1/2}m} 
> (1- 5\delta^{1/8}) \sqrt{m}. \]
Note that $G'$ is bipartite. 
By Lemma \ref{lem:close-to-Kab}, 
there exist two disjoint vertex subsets $U,V$ such that $d(G',K_{U,V})\le \varepsilon m/2$. 
Thus, we have 
\[ d(G,K_{U,V}) \le d(G,G') + d(G',K_{U,V}) \le 5\delta^{1/2}m + \varepsilon m/2 \le \varepsilon m. \]
This completes the proof.
\end{proof}

We now finish the proof of Theorem \ref{thm-edge-spectral-stability} 
for the case $\chi (F)=r+1=3$.

\begin{proof}[{\bf Proof of part (a) of Theorem \ref{thm-edge-spectral-stability}}]
    Let $\delta_0$ be the value of $\delta$ from Theorem \ref{thm-stability-triangle} with $\varepsilon' = \varepsilon / 3$, and let $\delta_1$ be the value of $\delta$ from Lemma \ref{lem:condition-max} with $\varepsilon' = \varepsilon / 2$. Without loss of generality, we may assume that $\delta_0 < \min\{10^{-10}, \varepsilon/10 \}$. We take $\delta = \min\{\delta_0^3, \delta_1^3\}$. Let $G$ be an $F$-free graph with $m\ge m_0$ edges and $\lambda^2 (G)\ge ( \frac{1}{2} - \delta) 2m$. 

    Applying Lemma \ref{lem:removing-bad-edges} to $G$, we find a subgraph $G'$ of $G$ with $m'\ge (1 - 4 \delta) m$ edges such that the Perron--Frobenius eigenvector $\bm{x}'=(x_i')_{i\in V(G')}$ satisfies $x'_i x'_j \geq \delta m^{-1/2}$ for every  $ij\in E(G')$. 
    Hence, we have $d(G,G')\le 4\delta m$. 
   After removing all isolated vertices from $G'$, we let $V'$ be its vertex set. By Lemma \ref{lem:remove-edges}, we have 
    $$\lambda^2(G') > \lambda^2(G) - 4\sqrt{4\delta} m > (1- 10\sqrt{\delta}) m.$$
    We now divide the remaining proof into two cases.

\smallskip
    
    \noindent\textbf{Case 1.} $\max\{x_v' : v\in V(G')\} > \delta_1^{-1} m^{-1/4} $.

We apply Lemma \ref{lem:condition-max} to $G'$ with parameter $m$, instead of $m'$. 
Indeed $e(G')\le m$. 
    The above bound yields  
    $\lambda^2(G')>(1-10\sqrt\delta)m\ge(1-\delta_1)m$, 
    and the present case gives
    $\max\{x'_v:v\in V(G')\}>\delta_1^{-1}m^{-1/4}$.
    Hence, there exist disjoint vertex subsets $U,V$
    with $d(G',K_{U,V})\le \varepsilon m/2$. 
    Thus, it follows that 
    $d(G,K_{U, V}) \le d(G,G') + d(G' ,K_{U, V}) \le 4 \delta m + \varepsilon m/2 \le \varepsilon m$ as desired. 

\smallskip
    
    \noindent\textbf{Case 2.} $\max\{x_v' : v\in V(G')\} \leq \delta_1^{-1} m^{-1/4}$. 
    
    In this case, we apply an argument similar to that used in the case $r \geq 3$. 
    First of all, we have 
    $\min\{x_v': v\in V(G')\} \ge \delta \delta_1 m^{-1/4}$. 
    Observe that $1=\sum_{i\in V'} (x_i')^2\ge |V'| \cdot (\delta \delta_1)^2 m^{-1/2}$, which gives 
     $|V'| \le (\delta \delta_1)^{-2} m^{1/2} $. 
    By Lemma \ref{lem-EFR}, there exists $m_0$ such that for $m\ge m_0$, 
    we can remove at most 
    $(\delta \delta_1)^5|V'|^2 \le 
    \delta m$ edges from $G'$ to get a 
    triangle-free graph $G''$. By Lemma 
    \ref{lem:remove-edges}, 
    \[  \lambda^2(G'') > \lambda^2(G) - 
    4\sqrt{5\delta} m > (1- 11\sqrt{\delta})m. \]
    Applying Theorem \ref{thm-stability-triangle}, there exist disjoint vertex subsets $U, V$ such that      
    $d(G'',K_{U, V}) \le \varepsilon m/3$. Thus, we get  
    $ d(G,K_{U, V}) \le d(G,G') + d(G',G'') + 
    d(G'',K_{U, V}) \le \varepsilon m$, as needed. 
\end{proof}

\section{Maximizing common neighbors}

\label{sec:common}

We first give a random construction showing that the leading constant $\tfrac12$ in Theorem \ref{thm-common} is best possible. Let $0<\varepsilon<\tfrac14$ and consider the
binomial random graph $G(n,\tfrac12+\varepsilon)$. We claim that for $n$ large, with high probability both of the following hold:
\begin{enumerate}
    \item[(1)] 
    $G$ has at least $(\tfrac12+\tfrac{\varepsilon}{2})\binom{n}{2}$ edges;
   
    \item[(2)] 
    every two vertices have at most $(\tfrac12+2\varepsilon)^2 n$ common neighbors.
\end{enumerate}
For (1), the number of edges is a sum of $\binom{n}{2}$
independent indicators with expectation 
$(\tfrac12+\varepsilon)\binom{n}{2}$, so by Chernoff's bound,  it
is at least $(\tfrac12+\tfrac{\varepsilon}{2})\binom{n}{2}$ with
probability $1-e^{-\Omega(n^2)}$. 
For (2), fix two vertices $u$ and $v$, 
their codegree is a sum of $n-2$ independent indicators with 
expectation $(\tfrac12+\varepsilon)^2(n-2)$. Since 
$(\tfrac12+2\varepsilon)^2-(\tfrac12+\varepsilon)^2
=\varepsilon+3\varepsilon^2>0$, Chernoff's bound gives the codegree of $\{u,v\}$ 
at most $(\tfrac12+2\varepsilon)^2 n$ with probability
$1-e^{-\Omega(n)}$. Moreover, a union bound over the $\binom{n}{2}$
pairs of vertices makes this hold for all pairs simultaneously with high
probability. 

Fix a graph $G$ satisfying (1) and (2), and set $m=e(G)$. By
(1), we have $m\ge (\tfrac12+\tfrac{\varepsilon}{2})\binom{n}{2}
>\tfrac{n^2}{4}$ for $n$ large, 
so $\tfrac{2m}{n}>\sqrt{m}$. Hence, the Rayleigh quotient gives
$\lambda(G)\ge \tfrac{2m}{n} >\sqrt{m}$ and $G$
satisfies the hypothesis of Theorem \ref{thm-common}. On the
other hand, if $K_{2,t}\subseteq G$, then the two vertices on
its side of size $2$ share $t$ common neighbors. By (2) and
$n<2\sqrt{m}$, we obtain 
\[ t\le \Big(\tfrac12+2\varepsilon\Big)^2 n
   < \Big(\tfrac12+2\varepsilon\Big)^2\cdot 2\sqrt{m}
   = \Big(\tfrac12+4\varepsilon+8\varepsilon^2\Big)\sqrt{m}
   \le \Big(\tfrac12+6\varepsilon\Big)\sqrt{m}, \]
where the last inequality holds since $\varepsilon < \frac{1}{4}$ and $8\varepsilon^2\le 2\varepsilon$.
Therefore $G$ has spectral radius larger than $\sqrt{m}$ yet
contains no $K_{2,t}$ with $t>(\tfrac12+6\varepsilon)\sqrt{m}$.
As $\varepsilon>0$ is arbitrary, the coefficient $\tfrac12$ in
Theorem \ref{thm-common} cannot be replaced by any larger
constant, so the bound is best possible. 

\medskip 
To prove~Theorem \ref{thm-common}, we need the following lemma.  

\begin{lemma} \label{lem-lower-bound}
    Let $G=(V,E)$ be a graph with $m$ edges and 
    $\bm{x}$ be the unit Perron--Frobenius eigenvector of $G$. Then 
    $$\sum_{i,j\in V} |N(i)\cap N(j)|\cdot x_i^2 x_j^2 \geq \frac{\lambda^3(G)}{2m}.$$
\end{lemma}

The bound in Lemma \ref{lem-lower-bound} 
is the best possible when we take $G=K_{\sqrt{m},\sqrt{m}}$. 

\begin{proof}
Let $\lambda =\lambda (G)$ and $C_{i, j}$ be the number of common neighbors of $i$ and $j$. In particular, we have that  $C_{i, i} $ is the degree $ d_i$. 
We can write
$$\sum_{i, j\in V} C_{i, j} x_i^2 x_j^2 
= \sum_{i\in V} x_i^2 \sum_{k \in N(i)} \sum_{j \in N(k)} x_j^2.$$
By the Cauchy--Schwarz inequality, we have
$$\sum_{j \in N(k)} x_j^2 \geq \frac{1}{d_k} \left(\sum_{j \in N(k)} x_j\right)^2 = \frac{\lambda^2}{d_k} x_k^2.$$ 
Thus, we obtain 
$$\sum_{i, j\in V} C_{i, j} x_i^2 x_j^2 \geq 
\lambda^2 \sum_{i\in V} \sum_{k \in N(i)} \frac{x_i^2 x_k^2}{d_k}
= \lambda^2 \sum_{ik \in E} x_i^2 x_k^2
  \left( \frac{1}{d_i} + \frac{1}{d_k} \right),$$
where the last equality collects, for each edge $ik$, its two
ordered contributions $(i,k)$ and $(k,i)$. By the AM-GM
inequality, we have $\frac{1}{d_i}+\frac{1}{d_k}\ge\frac{2}{\sqrt{d_id_k}}$,
and hence 
\begin{equation} \label{eq-low}
    \sum_{i, j\in V} C_{i, j} x_i^2 x_j^2 \geq 2\lambda^2  \sum_{ik \in E} \frac{x_i^2 x_k^2}{\sqrt{d_i d_k}}.
\end{equation} 
By the Cauchy--Schwarz inequality again, we have
\begin{equation*} 
    \sum_{ik \in E} \frac{x_i^2 x_k^2}{\sqrt{d_i d_k}} \geq \frac{\left(\sum_{i k \in E} x_i x_k\right)^2}{\sum_{ik \in E} \sqrt{d_i d_k}}.
\end{equation*}
On the one hand, it follows from the definition that $\sum_{ik \in E} x_i x_k = \lambda / 2.$
On the other hand, the vector $ (\sqrt{d_i})_{i=1}^n$ has length $\sqrt{2m}$. Therefore, the Rayleigh formula gives 
$$\sum_{ik \in E} \sqrt{d_i d_k} \leq 
\frac{\lambda}{2} \sum_{i=1}^n d_i =
\frac{\lambda}{2} \cdot 2m .$$
Thus, one more application of the Cauchy--Schwarz inequality yields
\begin{equation} \label{eq-bow}
   \sum_{ik \in E} \frac{x_i^2 x_k^2}{\sqrt{d_i d_k}} \geq \frac{\lambda}{4m}. 
\end{equation}
Combining  (\ref{eq-low}) 
with (\ref{eq-bow}), we conclude that $\sum_{i, j\in V} C_{i, j} x_i^2 x_j^2 \geq \frac{\lambda^3}{2m}$ as desired.
\end{proof}

Moreover, we present an alternative proof, which may be of independent interest.  

\begin{proof}[Second proof of Lemma \ref{lem-lower-bound}]
We may assume $m\geq 1$, so $\lambda:=\lambda(G)>0$.
If $d_u=0$ for some vertex $u$, then
$\lambda x_u=\sum_{v\in N(u)}x_v=0$ gives $x_u=0$, so
such vertices contribute nothing to either side and we
may assume $d_u\geq 1$ for every $u\in V$.

For every $k\in V$, set $s_k:=\sum_{i\in N(k)}x_i^2$. Since
$|N(i)\cap N(j)|=\sum_{k\in V}a_{ik}a_{kj}$, we get
\begin{align*}
\sum_{i,j\in V}|N(i)\cap N(j)|\,x_i^2x_j^2
=\sum_{k\in V}\Bigg(\sum_{i\in N(k)}x_i^2\Bigg)
  \Bigg(\sum_{j\in N(k)}x_j^2\Bigg) =\sum_{k\in V}s_k^2.
\end{align*}
Let $T:=\sum_{k\in V}\sqrt{d_k}\,s_k$. By the
Cauchy--Schwarz inequality, we have 
\[ 2m\sum_{k\in V}s_k^2  = \Bigg(\sum_{k\in V}d_k\Bigg)
\Bigg(\sum_{k\in V}s_k^2\Bigg) \geq T^2.
\]
Hence, it suffices to show that $T^2 \ge \lambda^3$, i.e., $T^{2/3} \geq \lambda$.

Summing over ordered pairs $(u,v)$ with $uv\in E(G)$,
we obtain 
\[
T=\sum_{k\in V}\sqrt{d_k}\sum_{i\in N(k)}x_i^2
 =\sum_{(u,v)\,:\,uv\in E}\sqrt{d_v}\,x_u^2 .
\]
The H\"older inequality, applied over all
ordered pairs $(u,v)$ with $uv\in E(G)$, yields
\begin{align*}
\lambda =\sum_{(u,v)\,:\,uv\in E}x_ux_v \leq\Bigg(\sum_{(u,v)}\sqrt{d_v}\,x_u^2\Bigg)^{1/3}
 \Bigg(\sum_{(u,v)}\sqrt{d_u}\,x_v^2\Bigg)^{1/3} 
 \Bigg(\sum_{(u,v)}
 \frac{x_ux_v}{\sqrt{d_ud_v}}\Bigg)^{1/3}.
\end{align*}
The first two factors are both equal to $T^{1/3}$.
For the third factor, using the AM-GM inequality $2ab\leq a^2+b^2$ with $a=x_u/\sqrt{d_u}$ and $b=x_v/\sqrt{d_v}$ gives
\[
\sum_{(u,v)\,:\,uv\in E}\frac{x_ux_v}{\sqrt{d_ud_v}}
=\sum_{uv\in E}\frac{2x_ux_v}{\sqrt{d_ud_v}}
\leq\sum_{uv\in E}
\Bigg(\frac{x_u^2}{d_u}+\frac{x_v^2}{d_v}\Bigg)
=\sum_{u\in V}x_u^2=1 .
\]
Therefore $\lambda\leq T^{2/3}$, that is,
$T^2 \geq\lambda^{3}$. Combining the above
inequalities, we obtain
\[
\sum_{i,j\in V}|N(i)\cap N(j)|\,x_i^2x_j^2
=\sum_{k\in V}s_k^2
\geq\frac{T^2}{2m}
\geq\frac{\lambda^3}{2m},
\]
as desired.
\end{proof}

Now, we are in a position to prove Theorem \ref{thm-common}. 

\subsection{Proof of Theorem \ref{thm-common}}

Let $C_{i,j}=|N(i)\cap N(j)|$. Since $\lambda (G)> \sqrt{m}$, Lemma \ref{lem-lower-bound} gives 
    $$\sum_{i\in V} d_i x_i^4 + \sum_{i \neq j} C_{i, j} x_i^2 x_j^2 > \frac{\sqrt{m}}{2}.$$
    We denote $t := \max_{i \neq j} C_{i, j}$. Then 
    $$\sum_{i \neq j} C_{i, j} x_i^2 x_j^2 \leq t \sum_{i \neq j} x_i^2 x_j^2 \leq t.$$
    Note that 
    $$\sum_{i\in V} d_i x_i^4 \leq 2m \cdot (\max_{i\in V} x_i)^4. $$
    Therefore, we conclude that 
    \begin{equation} \label{eq-t-max}
        t + 2m \cdot (\max_{i\in V} x_i)^4 > \frac{\sqrt{m}}{2}.
    \end{equation}
    
    Let $\varepsilon \in (0,0.01)$ be a fixed  number, and let  
    $\delta = \delta (\varepsilon) >0$ be the constant supplied by  
    the spectral stability Lemma \ref{lem:condition-max}. 
    Since $\varepsilon$ is an
    absolute constant, so is $\delta$, and
    $2\delta^{-4}=O(1)$. 
    
    If $\max_{i\in V} \{x_i\} \le \delta^{-1}m^{-1/4}$, 
    then (\ref{eq-t-max}) implies $t> \sqrt{m}/2 - 2\delta^{-4}$, as desired. 
    
   Now we assume $\max_{i\in V} \{x_i\} > \delta^{-1} m^{-1/4} $, so that Lemma \ref{lem:condition-max} is applicable. 
    There exist disjoint vertex subsets $A,B$ of $G$ such that 
    $d(G,K_{A,B})\le \varepsilon m$, which means that the number of edges with both endpoints in $A$ or both in $B$, the number of missing edges between $A$ and $B$, and the number of edges with at least one endpoint outside $A\cup B$ together amount to  at most $\varepsilon m$. 
    Then $e(A)+e(B)\le \varepsilon m$ and 
    \[ (1- \varepsilon)m \le  |A||B|  \le (1+ \varepsilon )m. \]  
    We may assume that $|A|\le |B|$. 
    Then $|A| \le \sqrt{(1+\varepsilon)m}$. 

We claim that $ |A|\ge 2$. Suppose on the contrary that 
$|A|=1$. Using Lemma \ref{lem-Delta}, we know that 
$ e(G)\le \Delta (G) + \varepsilon m \le m/2 + m^{0.99} + \varepsilon m <m$,  a contradiction. 
 Thus, we have $|A|\ge 2$.

Next, we find two vertices in $A$ with many
common neighbors in $B$. 
Suppose on the contrary that every pair $\{i,j\}\subseteq A$
has fewer than $(1-10\varepsilon)m/|A|$ common neighbors in
$B$. Then 
$$\sum_{\{i,j\}\subseteq A}|N(i)\cap N(j)\cap B| < \binom{|A|}{2}\cdot\frac{(1-10\varepsilon)m}{|A|}
 =\frac{(|A|-1)(1-10\varepsilon)m}{2}.$$
 For each vertex $k\in B$, we write
$d_A(k):=|N(k)\cap A|$. Since at most $\varepsilon m$ edges of
$G$ lie outside $K_{A,B}$, we have
$\sum_{k\in B}d_A(k)=e(A,B)\ge (1-\varepsilon)m$. 
By double counting, we have 
\[ \sum_{\{i,j\}\subseteq A}|N(i)\cap N(j)\cap B|
   = \sum_{k\in B}\binom{d_A(k)}{2}
   \ge |B|\binom{e(A,B)/|B|}{2}, \]
where the inequality uses the convexity of the function
$\binom{x}{2}$.

 Combining with
$e(A,B)\ge (1-\varepsilon)m$, we obtain
\[ (|A|-1)(1-10\varepsilon)
   > (1-\varepsilon) 
   \left(\frac{(1-\varepsilon)m}{|B|}-1\right). \]
Using $m/|B|\ge |A|/(1+\varepsilon)$ (from
$|A||B|\le (1+\varepsilon)m$) on the right-hand side and
simplifying, this forces
$|A|(7+11\varepsilon)<9(1+\varepsilon)$, i.e.\ $|A|<2$,
which contradicts $|A|\ge 2$. 
Hence, some pair $\{i,j\}\subseteq A$ has at least
$ \frac{(1-10\varepsilon)m}{|A|}$ common neighbors in $B$.  
Since
$|A|\le\sqrt{(1+\varepsilon)m}$, we conclude
$$ t\ge \frac{(1-10\varepsilon)m}{|A|}
    \ge \frac{(1-10\varepsilon)m}{\sqrt{(1+\varepsilon)m}}
    \ge (1-5\sqrt{\varepsilon})\sqrt{m}.$$ 
  Since $\varepsilon\le 0.01$ is fixed, this reads
$t\ge\tfrac12\sqrt m$. Hence,
 in either case
$t\ge\tfrac12\sqrt m-O(1)$, and the two vertices of
this pair together with their $t$ common neighbors
form the required copy of $K_{2,t}$.

\section{Related dense substructures}
\label{sec:concluding}
In this paper, we first established that for any graph $F$ with $\chi(F) = r+1\ge 3$, every $m$-edge $F$-free graph $G$ satisfies $\lambda^2(G)\le \big(1-\frac{1}{r} + o(1) \big)2m$; see Theorem \ref{thm-edge-spectral-ESS}. 
This extends the classical Erd\H{o}s--Stone--Simonovits theorem to the edge-spectral setting and unifies spectral Tur\'{a}n-type results in a general framework. 
Second, we generalized the edge-spectral version to the walk-spectral setting for any walks of length $\ell$; see Theorem \ref{thm:walk}. 
Third, we proved the edge-spectral stability result by  showing that any $F$-free graph $G$ with $\lambda^2(G)\ge (1- \frac{1}{r} - o(1))2m$ is structurally close to a complete bipartite graph for $r=2$, or an $r$-partite  Tur\'{a}n graph for $r\ge 3$; see Theorem \ref{thm-edge-spectral-stability}. This unifies both the classical Erd\H{o}s--Simonovits stability and its vertex-spectral version of Nikiforov. 
As an application, we showed that 
any $m$-edge graph $G$ with $\lambda (G) >\sqrt{m}$ 
contains $K_{2,t}$ as a subgraph, where $t\ge \frac{1}{2}\sqrt{m} - O(1)$ is optimal as achieved by a random graph construction; see Theorem \ref{thm-common}. 
This significantly improves a recent bound of Zhai, Lin and Shu.

For any $m$-edge graph $G$ with spectral radius at least $(1 + \varepsilon)\sqrt{m}$, our technique also shows that for any fixed integer $s\ge 2$, there exists a copy of $K_{s,t}$ with $t=\Omega(\sqrt{m})$ in $G$.

\begin{theorem} \label{thm-squar-root}
For every $s\ge 2$ and $\varepsilon >0$, 
there exists $\delta = \delta (s,\varepsilon)>0$ such that 
if $G$ is an $m$-edge graph with 
$\lambda (G)\ge (1+ \varepsilon )\sqrt{m}$, then $G$ contains a copy 
of $K_{s,t}$ with $t\ge \delta \sqrt{m}$. 
\end{theorem}

We need a variant of K\H{o}v\'{a}ri--S\'{o}s--Tur\'an's theorem (see \cite{Bollobas78}): 
If $\varepsilon >0$ and $G$ is an $n$-vertex graph with $e(G)\ge \varepsilon n^2,$
then $G$ has a copy of 
$K_{s,t}$ with $s\ge 1$ and $t= (\varepsilon/{2s})^s n$. 
In the proof of Theorem \ref{thm-edge-spectral-ESS}, we choose a minimal counterexample $G$, and show that  if $uv\in E(G)$ is light, i.e., $x_ux_v \le 1/(8\sqrt{m})$, then removing the single edge $uv$ would create an example with smaller size, a contradiction. Hence $x_ux_v > 1/(8\sqrt{m})$ for every $uv\in E(G)$, which together with Lemma \ref{lem:max} forces $G$ to be dense. Here, we use a slightly different approach.  we discard all {\it light} edges of small weight at once, and pass to a \emph{heavy} subgraph $H$. Then we show that $|V(H)|=O(\sqrt{m})$ and $e(H)=\Omega (m)$. 
So $H$ is dense on $\Theta(\sqrt m)$
vertices, and hence applying the variant of KST theorem to $H$ proves Theorem \ref{thm-squar-root}. Producing a dense subgraph directly, this method
applies uniformly to existence, to balanced complete
bipartite subgraphs, and to supersaturation.

\begin{proof}[Proof of Theorem \ref{thm-squar-root}]
Let $\bm{x}$ be a unit Perron eigenvector of $G$, and write
$M:=\max_{v\in V}x_v$. Since
$\lambda^2(G)\ge (1+\varepsilon)^2 m\ge (1+2\varepsilon)m$,
Lemma \ref{lem:max} gives $M< 1/(16\varepsilon^4 m^{1/4})$.

Call an edge $ij\in E(G)$ \emph{heavy} if
$x_ix_j\ge 1/(8\sqrt m)$, and \emph{light} otherwise.
Let $H$ be the spanning subgraph of $G$ formed by the
heavy edges, and let $h$ be the number of its non-isolated vertices.

We first bound $h$. If $ij$ is heavy, then
\[
  x_j\ge \frac{1/(8\sqrt m)}{x_i}
       \ge \frac{1/(8\sqrt m)}{M}
       > \frac{2\varepsilon^4}{m^{1/4}}.
\]
Hence, every vertex of $H$ has weight more than $2\varepsilon^4 m^{-1/4}$, so
\[
  1=\sum_{v\in V}x_v^2\ge \sum_{v\in V(H)}x_v^2
   > \frac{4\varepsilon^8 h}{\sqrt m},
\]
giving $h < \sqrt m/(4\varepsilon^8)$, that is, $m> 16\varepsilon^{16}h^2$.

Next we bound $e(H)$ from below. Splitting the Rayleigh quotient 
$\lambda (G)=\sum_{ij\in E}2x_ix_j$ into heavy and light
edges, the light part contributes at most
\[
  \sum_{ij\text{ light}}2x_ix_j
   < |E(G)|\cdot\frac{1}{4\sqrt m}
   \le \frac{\sqrt m}{4},
\]
while each heavy edge contributes
$2x_ix_j\le 2M^2< 1/(128\varepsilon^8\sqrt m)$. Hence
\[
  (1+\varepsilon)\sqrt m\le \lambda (G)
   < \frac{\sqrt m}{4} + \frac{e(H)}{128\varepsilon^8\sqrt m},
\]
so $e(H)> 128\varepsilon^8(\tfrac34+\varepsilon)m
> 96\varepsilon^8 m$. Together with
$m> 16\varepsilon^{16}h^2$, we obtain 
\[
  e(H)> 96\varepsilon^8 m
      > 96\varepsilon^8\cdot 16\varepsilon^{16} h^2
      = 1536\,\varepsilon^{24} h^2 .
\]
By the variant of the
K\H{o}v\'ari--S\'os--Tur\'an theorem, $H$ contains a copy
of $K_{s,t}$ with
$t\ge (1536\varepsilon^{24}/2s)^{s} h$. Finally
$e(H)\le \binom{h}{2}<h^2/2$ and
$e(H)>96\varepsilon^8 m$ give
$h >\sqrt{192}\,\varepsilon^4\sqrt m$, so
\[
  t> \Big(\tfrac{768\varepsilon^{24}}{s}\Big)^{\!s}
     \sqrt{192}\,\varepsilon^4\,\sqrt m
   =: \delta\sqrt m .
\]
Since $H\subseteq G$, this is the desired copy in $G$.
\end{proof}

 A similar argument of Theorem \ref{thm-squar-root} gives the following result.  
 
 \begin{theorem} 
For every $\varepsilon >0$, 
there is a constant $\delta = \delta (\varepsilon) >0$ such that 
if $G$ is an $m$-edge graph with 
$\lambda (G)\ge (1+ \varepsilon )\sqrt{m}$, then $G$ contains a copy 
of $K_{t,t}$ with $t= \delta \log (m)$. 
\end{theorem}

\begin{proof}
As in the proof of Theorem \ref{thm-squar-root}, fix a unit Perron eigenvector of $G$, 
discard all light edges of $G$, and pass to the heavy subgraph $H$: the spanning subgraph formed by all edges $ij$ with $x_ix_j\ge 1/(8\sqrt m)$, restricted to its $h$ non-isolated vertices. The same estimates give
$ h < \frac{\sqrt m}{4\varepsilon^8}$ and 
  $e(H)> 96\varepsilon^8 m $. 
These bounds
yield the  edge density
${e(H)}/{\binom{h}{2}}
   > {96\varepsilon^8 m}/{h^2}
   > c\,\varepsilon^{24}$ with $c:=1536$. 
   Moreover, $h \ge \sqrt{2e(H)}\ge \sqrt{192} \, \varepsilon^4\sqrt m$, so that
$\log h \ge \tfrac12\log (m) - O(1)$.

Applying the following 
variant of the K\H{o}v\'ari--S\'os--Tur\'an theorem (see, e.g., \cite{Bollobas78}): if an $h$-vertex graph $H$ has edge density at least
$\eta$, then $H$ contains a copy of $K_{t,t}$ with
$t=\lfloor(\log h)/(2\log(1/\eta))\rfloor$. 
Setting  $\eta=c \,\varepsilon^{24}$, we obtain a
copy of $K_{t,t}$ in $H$, so in $G$, with
\[
  t\ge \frac{\log h}{2\log(1/ (c\,\varepsilon^{24}))}
     \ge \frac{\log m}{4\log(1/(c\,\varepsilon^{24}))}
        - O(1) .
\]
Setting
$\delta:=\big(8\log(1/(c\,\varepsilon^{24}))\big)^{-1}$
gives $t\ge\delta\log (m)$ for all large $m$, as desired.
\end{proof}

\noindent 
{\it Remark.} 
The $\log (m)$ bound is tight by considering the random graph $G(n, p)$, where each edge is presented independently with probability $p$. By a standard probabilistic argument, for any $p\in (0,1)$, there exists an $n$-vertex graph $G$ with $e(G)\ge p{n \choose 2}$, but $G$ does not contain $K_{t,t}$ with $t=\frac{3\log n}{\log (1/p)}$. Taking $p=\frac{2}{3}$, we have $\lambda (G) > 1.1 \sqrt{m}$, but $G$ has no $K_{t,t}$ with $t=8\log (m)$.  

\smallskip
In addition, it is well-known that 
any $n$-vertex graph with $\varepsilon n^2$ edges contains $\Omega_{t,\varepsilon} (n^{2t})$ copies of $K_{t,t}$. 
Using a similar framework as in the above, we obtain the following supersaturation. 

\begin{theorem}
    For every $t\ge 2$ and $\varepsilon >0$, 
    there is a constant $\delta =\delta (t,\varepsilon)$ such that if $G$ is an $m$-edge graph with $\lambda (G) \ge (1+ \varepsilon) \sqrt{m}$, then $G$ contains at least $\delta m^t$ copies of $K_{t,t}$. 
\end{theorem}

\medskip 
There are many promising directions to explore in the edge-spectral Tur\'{a}n-type problem, and we provide another example for cycles with consecutive lengths. 

\begin{theorem}
    For any $\varepsilon >0$, there exists $\delta = \delta (\varepsilon) >0$ such that if $G$ is an $m$-edge graph with $\lambda (G)\ge (1+ \varepsilon )\sqrt{m}$, then $G$ contains a cycle of length $\ell$ for every $3\le \ell \le \delta \sqrt{m}$. 
\end{theorem}

\begin{proof}
    We denote $\delta =2 \varepsilon$ and $\lambda = \lambda (G)$. 
    Then $\lambda (G)\ge (1+\varepsilon)\sqrt{m}$ gives 
    $\mu :=\lambda^2 - \varepsilon \sqrt{m} 
    \lambda > m$. Now, we consider the matrix $T:=A_G^2 - \varepsilon \sqrt{m} A_G$. 
    Let $\bm{x}$ be a nonnegative unit Perron eigenvector of $A_G$ corresponding to 
$\lambda$. Since $T$ shares the eigenvector $\bm{x}$ with
$T\bm{x}=\mu  \bm{x}$, writing
$r_i:=\sum_{j}t_{ij}$ and using that $T$ is symmetric, it follows that 
$
 \mu \sum_{i}x_i
   = \langle \mu \bm{x}, \bm{1} \rangle  
   = \langle T\bm{x}, \bm{1}\rangle = 
   \langle \bm{x}, T\bm{1}\rangle
   =\sum_i  x_i r_i $.
As $\bm{x} \ge 0$, this gives
$\lambda^2-\varepsilon\sqrt m\,\lambda
=\big(\sum_i x_i r_i\big)/\big(\sum_i x_i\big)
\le \max_{i\in V} r_i$.  
    On the other hand, 
 \[ \sum_{j\in V}t_{ij} =
    \sum_{k\sim i} \sum_{j\sim k} 1 - \varepsilon \sqrt{m} \sum_{j\sim i} 1 
    \le m +e(N(i)) - \varepsilon \sqrt{m} |N(i)|. \]
    Therefore, there exists a vertex  
    $i\in V$ such that $e(N(i)) > \varepsilon \sqrt{m} |N(i)|$. 
    By the Erd\H{o}s--Gallai theorem (see \cite{Bollobas78}), we find a path $P_{t+1}$ in $N(i)$ with length $t := 2 \lfloor \varepsilon \sqrt{m} \rfloor +1$. So $G$ contains a copy of $K_1\vee P_{t+1}$, which implies that $G$ contains all cycles $C_{\ell}$ with $3\le \ell \le 2\varepsilon \sqrt{m}$, as needed.     
\end{proof}

The above bound $\Omega (\sqrt{m})$ on $\ell$ is optimal up to a constant factor.

\section*{Declaration of competing interest}
The authors declare that they have no conflicts of interest to this work.

\section*{Data availability}
No data was used for the research described in the article.

\section*{Acknowledgements}   
This paper is dedicated to Vladimir Nikiforov, whose beautiful work on spectral extremal graph theory has inspired the authors. 
This work was initiated after the 1st IBS ECOPRO student research program in fall 2023. The first and third authors would like to thank the ECOPRO group for hosting them. The authors also thank Prof. Hao Huang for motivating this project.

\appendix

\section{Alternative proof via hypergraph theory} 

\label{sec-A}

In this section, we provide an alternative proof of the
edge-spectral stability for cliques (Theorem
\ref{thm-stability-clique-4}), 
using the Kruskal--Katona type 
stability of Liu and Mukherjee \cite{LM2023}. 
To begin with, we need some notation. 
We denote by $k_r(G)$ the number of cliques of order $r$ in $G$.  
For example, we have $k_1(G)=n$ and $k_2(G)=m$. 
The regular $r$-partite Tur\'an graph with $m$ edges has partite size $\mu:=\sqrt{\frac{2m}{r(r-1)}}$, and the number of $i$-cliques is $k_i=\binom{r}{i}\mu^i$ 
for every $i$. The idea is as follows.

\begin{enumerate}
    \item If $\lambda(G)$ is close to the extremal value
    $\sqrt{(1 \!-\! \frac{1}{r})2m}= (r-1)\mu$, then the number $k_r(G)$ of
    $r$-cliques is close to its maximum $\mu^r$ in the
    S\'os--Straus theorem.
    \item Let $\mathcal{H}$ be the $r$-uniform hypergraph whose
    edges are the $r$-cliques of $G$. By the stability theorem of
    Liu and Mukherjee, $\mathcal{H}$ is close to a complete
    $r$-partite $r$-graph.
    \item Returning to graphs, this forces $G$ to be close to a
    Tur\'an graph $T_{C,r}$, as desired.
\end{enumerate}

\begin{theorem}[Nikiforov \cite{Niki2002}] 
\label{thm-Niki-2002-p}
Let $G$ be a $K_{\ell+1}$-free graph. Then 
\begin{equation} \label{eq-Niki-2002-clique}
 \lambda^{\ell} (G) \le 
\sum_{i=2}^{\ell} (i-1)k_i(G) \lambda (G)^{\ell -i},
\end{equation}
\end{theorem}

\begin{theorem}[S\'{o}s--Straus \cite{SS1982}; 
Fisher--Ryan \cite{FR1992}]   
\label{thm-SS-FR}
Let $G$ be a $K_{\ell+1}$-free graph on $n$ vertices. 
For every $i\in [\ell]$, let $k_i$ denote the number of copies of $K_i$ 
in $G$. Then 
\begin{equation} \label{eq-SS-FR}
  \left( \frac{k_{\ell}}{{\ell \choose \ell}}\right)^{1/\ell} 
\le  \left( \frac{k_{\ell-1}}{{\ell \choose \ell-1}}\right)^{1/(\ell-1)} 
\le \cdots \le \left( \frac{k_2}{{\ell \choose 2}}\right)^{1/2} 
\le  \frac{k_1}{{\ell \choose 1}}. 
\end{equation} 
\end{theorem}

At the end of Section~3 in \cite{Niki2002}, Nikiforov
remarked that ``{\it it is possible to obtain the bound
\eqref{eq-Niki-2002-cpc} from Theorem
\ref{thm-Niki-2002-p} and \eqref{eq-SS-FR}. In this sense,
Theorem \ref{thm-Niki-2002-p} should be regarded as an
extension of \eqref{eq-Niki-2002-cpc}.}'' However, he only stated the possibility and did
not provide any details. To better understand the stability
of \eqref{eq-Niki-2002-cpc}, we first give a detailed proof
of it using Theorem~\ref{thm-Niki-2002-p} and \eqref{eq-SS-FR}.

\begin{proof}[{\bf Alternative proof of Nikiforov's bound (\ref{eq-Niki-2002-cpc})}] 
Write $\lambda=\lambda(G)$ and $k_i=k_i(G)$ for short. If
$m=0$, then $\lambda=0$ and the bound is trivial. 
So we assume that 
$m\ge 1$, and hence $\lambda>0$. Recall that 
$\mu:=\bigl(k_2/\binom{r}{2}\bigr)^{1/2}
=\sqrt{ \frac{2m}{r(r-1)}}>0$. Since $G$ is $K_{r+1}$-free,
Theorem~\ref{thm-SS-FR} yields
\begin{equation} \label{eq-ki-c}
 k_i \le \binom{r}{i}\mu^{i}
 \quad \text{for every } 2\le i\le r.
\end{equation}
Applying Theorem~\ref{thm-Niki-2002-p} with $\ell=r$ and
using \eqref{eq-ki-c}, we get 
\[ \lambda^{r} \le \sum_{i=2}^{r}(i-1)k_i\lambda^{r-i}
 \le \sum_{i=2}^{r}(i-1)\binom{r}{i}\mu^{i}\lambda^{r-i}. \]
Dividing by $\lambda^{r}$ gives $\varphi( \frac{\mu}{\lambda} )\ge 1$,
where
\[ \varphi(x):=\sum_{i=2}^{r}(i-1)\binom{r}{i}x^{i}. \]
Note that $\varphi(x)$ is strictly increasing on $[0,\infty)$. We claim that
\begin{equation} \label{eq-phi-one}
 \varphi\Bigl(\frac{1}{r-1}\Bigr)=1.
\end{equation}
 The binomial identity and its derivative form give 
 $\sum_{i=0}^{r}\binom{r}{i}y^{i}=(1+y)^{r}$ and $ \sum_{i=0}^{r} i\binom{r}{i}y^{i}=ry(1+y)^{r-1}$. 
Put $y:=\tfrac{1}{r-1}$, we get $ry= \frac{r}{r-1}=1+y$. 
Subtracting the first from the second gives
$\sum_{i=0}^{r}(i-1)\binom{r}{i}y^{i}=0$. 
In the sum, 
the term $i=0$ contributes $-1$, and the term $i=1$ contributes $0$, so the sum over
$2\le i\le r$ equals $1$. This proves the claimed \eqref{eq-phi-one}. 

Since
$\varphi$ is strictly increasing and
$\varphi( \frac{\mu}{\lambda} )\ge 1=\varphi(\tfrac{1}{r-1})$ by \eqref{eq-phi-one}, we get
$ \frac{\mu}{\lambda} \ge \tfrac{1}{r-1}$, that is,
\[ \lambda\le (r-1)\mu=\sqrt{\Bigl(1-\frac{1}{r}\Bigr) 2m}. \]
This completes the proof of Nikiforov's bound (\ref{eq-Niki-2002-cpc}).
\end{proof}

From the above proof, we observe that if $\lambda (G)$ is close to $\sqrt{(1\!-\! \frac{1}{r})2m}$, then $k_i (G)$ is close to ${r \choose i}\mu^i$ for some $2\le i\le r$. To show the stability of (\ref{eq-Niki-2002-cpc}), we need to establish a stability 
result of Lemma \ref{thm-SS-FR} for our purpose. This can be deduced from a stability result on hypergraphs. 
Let $\mathcal{H}$ be an $r$-uniform hypergraph on vertex set $V$. 
The shadow of $\mathcal{H}$ is defined as 
$$ \partial \mathcal{H} = \left\{A \in \tbinom{V}{r-1} : \text{there is $B\in E(\mathcal{H})$ such that $A\subset B$}  \right\}.$$

In 2021, Liu and Mubayi \cite{LM2021} 
proved the following inequality. 

\begin{theorem}[Liu--Mubayi \cite{LM2021}] 
\label{thm-LM2021}
Let $\ell \ge r \ge 2$ and $\mathcal{H}$ be a $\mathcal{K}_{\ell +1}^{(r)}$-free 
$r$-graph. Then 
\[  \left( \frac{|\mathcal{H}|}{{\ell \choose r}} 
\right)^{{1}/{r}} \le 
\left( \frac{| \partial \mathcal{H}|}{{\ell \choose r-1}} \right)^{{1}/{(r-1)}} . \]
\end{theorem}

This can deduce the inequality in 
Theorem \ref{thm-SS-FR} by translating each $r$-clique of a graph into an edge of an $r$-graph.  
In 2023, Liu and Mukherjee \cite{LM2023} proved 
the following stability result.

\begin{theorem}[Liu--Mukherjee \cite{LM2023}] \label{thm-LM2022}
For every $\ell \ge r \ge 2$ and $\varepsilon >0$, 
there exist $n_0$ and $\delta >0$ such that 
the following holds for all numbers $x\ge n_0$. 
Suppose that $\mathcal{H}$ is an  $\mathcal{K}_{\ell +1}^{(r)}$-free $r$-graph with 
\[ |\partial \mathcal{H}| = {\ell \choose r \!-\!1} \left( \frac{x}{\ell}\right)^{r-1}~~\text{and}~~ 
|\mathcal{H}| \ge (1-\delta ){\ell \choose r} 
\left( \frac{x}{\ell}\right)^r. \]
Then there exists a vertex set $C\subseteq V(\mathcal{H})$ of size at most $\lceil x\rceil$
such that $d(\mathcal{H},\mathcal{T}_{C,\ell}^{(r)}) \le \varepsilon x^r$.  
\end{theorem}

\subsection{Second proof of Theorem \ref{thm-stability-clique-4}}

First, we show that a nearly extremal $\lambda (G)$ forces
$k_r(G)$ to be nearly maximal.

\begin{lemma} \label{lem-kr-lower}
    Let $r\ge 3$ and $0<\beta<1$. If $G$ is a $K_{r+1}$-free graph
    with $m$ edges and 
    $$\lambda(G)\ge (1-\delta)(r-1)\mu,$$ where
    $\delta:=\beta/r^{r-1}$, then $k_r(G)\ge (1-\beta)\mu^r$.
\end{lemma}

\begin{proof}
Recall that $\lambda=\lambda(G)$, $k_i=k_i(G)$   
and $\mu:=\big(m/\binom{r}{2}\big)^{1/2}$. 
    By Theorem \ref{thm-SS-FR} with $\ell=r$, for each
    $2\le i\le r-1$, we have 
    $\big(k_i/\binom{r}{i}\big)^{1/i}\le
    \big(k_2/\binom{r}{2}\big)^{1/2}
    :=\mu$, and hence 
    $k_i \le \binom{r}{i}\mu^i$. Substituting these into
Theorem~\ref{thm-Niki-2002-p} and keeping
    the top term $i=r$, we get 
    \[ \lambda^r \le
       \sum_{i=2}^{r-1}(i-1)\binom{r}{i}\mu^i\,\lambda^{r-i} 
       + (r-1)k_r . \]
  Thus, we obtain $(r-1)k_r \ge h(\lambda )$, where
    \[ h(x):=x^r-\sum_{i=2}^{r-1}(i-1)\binom{r}{i}\mu^i x^{r-i}. \]
    We next prove that
    \begin{equation} \label{eq-factor}
      h(x)=\big(x-(r-1)\mu\big)(x+\mu)^{r-1}+(r-1)\mu^r.
    \end{equation} 
Put $t:=x/\mu$ and
    $q(t):=t^r-\sum_{i=2}^{r-1}(i-1)\binom{r}{i}t^{r-i}$. 
    To prove (\ref{eq-factor}), 
    it suffices to prove that 
    $q(t)=(t-(r-1))(t+1)^{r-1} + (r-1)$. 
    The coefficient of $t^k$ in $(t-(r-1))(t+1)^{r-1}$ is
    $\binom{r-1}{k-1}-(r-1)\binom{r-1}{k}
    =\binom{r}{k}(k-r+1)$, which equals
    the coefficient of $t^k$ in $q(t)$ for every $1\le k\le r$. Moreover, both sides of the claimed identity have constant term $0$.  
    Hence $q(t)=(t-(r-1))(t+1)^{r-1} + (r-1)$. 
    
     If $\lambda\ge (r-1)\mu$, then \eqref{eq-factor} gives
    $h(\lambda)\ge (r-1)\mu^r$. If
    $(1-\delta)(r-1)\mu\le \lambda<(r-1)\mu$, then
    $-  \delta(r-1)\mu \le  \lambda - (r-1)\mu  < 0$ and
    $\lambda+\mu\le r\mu$, so \eqref{eq-factor} gives
    \[ h(\lambda)\ge 
       -\delta(r-1)\mu\cdot (r\mu)^{r-1} + (r-1)\mu^r 
       = (r-1)\mu^r\big(1-\delta r^{r-1}\big). \]
    In either case $(r-1)k_r \ge (r-1)\mu^r(1-\delta r^{r-1})$,
    whence $k_r \ge (1-\delta r^{r-1})\mu^r=(1-\beta)\mu^r$.
\end{proof}

Lemma \ref{lem-kr-lower} reduces the spectral stability of an $m$-edge $G$ to the clique stability. 
Next, we reduce the clique stability of $G$ to the Kruskal--Katona type stability of the associated $r$-uniform hypergraph.  The following lemma establishes a stability result of Theorem \ref{thm-SS-FR}. 

\begin{lemma} \label{lem-reduce-general}
Denote $\mu:=\bigl(m/\binom{r}{2}\bigr)^{1/2}$. 
    For every $r\ge 3$ and $0<\varepsilon<0.01$, there exist
    $\beta>0$ and $m_0$ such that if $G$ is a $K_{r+1}$-free graph
    with $m\ge m_0$ edges and 
    $$k_r(G)\ge (1-\beta)\mu^r,$$
     then
    $d(G,T_{C,r})\le \varepsilon m$ for some vertex set
    $C\subseteq V(G)$.
\end{lemma}

\begin{proof}
    Set $\gamma:=\varepsilon/2$ and
    $\beta:=\min\{\delta_0,\gamma^4\}$, where
    $\delta_0=\delta_0(\gamma^2/r^r)$ is the constant from
    Theorem \ref{thm-LM2022}. 
    Let $\mathcal{H}$ be the $r$-uniform hypergraph with
    $V(\mathcal{H})=V(G)$ whose edges are the $r$-cliques of $G$. 
    Then $|\mathcal{H}|=k_r(G)$. Each element of
    $\partial\mathcal{H}$ is an $(r-1)$-set spanning a $K_{r-1}$,
    so $|\partial\mathcal{H}|\le k_{r-1}(G)$. Since $G$ is
    $K_{r+1}$-free, the associated hypergraph $\mathcal{H}$ is $\mathcal{K}_{r+1}^{(r)}$-free.
    Set $x:=r\,( \frac{|\partial\mathcal{H}|}{r})^{1/(r-1)}$, so that
    $|\partial\mathcal{H}|=\binom{r}{r-1}( \frac{x}{r})^{r-1}$, which meets the
    shadow hypothesis of Theorem \ref{thm-LM2022} with $\ell=r$.

    By Theorem \ref{thm-SS-FR}, we have 
    $|\partial\mathcal{H}|\le k_{r-1}(G)\le r\mu^{r-1}$, which 
    together with the definition of $x$, yields  
    $x\le r\mu$ and $(|\partial\mathcal{H}|/r)^{r/(r-1)}\le\mu^r$. 
    Using $\beta\le\delta_0$, it follows that 
    \[ |\mathcal{H}|=k_r(G)\ge (1-\beta)\mu^r
       \ge (1-\delta_0)\big(|\partial\mathcal{H}|/r\big)^{r/(r-1)}
       =(1-\delta_0)(x/r)^r, \]
    which is the density hypothesis of Theorem \ref{thm-LM2022}. 
Moreover, by the Kruskal--Katona lower shadow bound
    (Theorem \ref{thm-LM2021}) and $|\mathcal{H}|=k_r(G)\ge
    (1-\beta)\mu^r$, we get 
    \[ (|\partial\mathcal{H}| / r)^{1/(r-1)} \ge\ 
    |\mathcal{H}|^{1/r}
       \ \ge\ (1-\beta)^{1/r}\mu, \]
    so that $x=r(|\partial\mathcal{H}|/r)^{1/(r-1)}
    \ge r(1-\beta)^{1/r}\mu$. Since $\beta$ is a constant
    and $r\mu=\sqrt{2m/(1-1/r)}\to\infty$,
    we have $x\ge n_0$ for $m$ large. 
    Applying Theorem \ref{thm-LM2022} 
    with $\varepsilon_{\mathrm{LM}}:=\gamma^2/r^r$ and using
    $x^r\le r^r\mu^r$, we obtain a set $C\subseteq V(G)$ with
    $|C|\le\lceil x\rceil\le\lceil r\mu\rceil
    =\lceil\sqrt{2m/(1-1/r)}\rceil$ such that
    \begin{equation} \label{eq-count-bad-copy}
     d(\mathcal{H},\mathcal{T}_{C,r}^{(r)})
       \le \varepsilon_{\mathrm{LM}}\,x^r\le \gamma^2\mu^r. \end{equation}

Let $C=V_1\sqcup\cdots\sqcup V_r$ be the balanced
    partition underlying the Tur\'an $r$-graph
    $\mathcal{T}_{C,r}^{(r)}$ produced by
    Theorem~\ref{thm-LM2022}, and let $T_{C,r}$ be the complete $r$-partite graph on these parts, 
    so that $\bigl||V_i|-|V_j|\bigr|\le 1$ and  $|C|/r-1\le |V_i|\le |C|/r+1$ for all $i$. 
    Call an $r$-clique of $G$ 
    \emph{good} if it meets each $V_i$ in exactly one vertex, and
    \emph{bad} otherwise. 
    The bound (\ref{eq-count-bad-copy}) says that $G$ has at most
    $\gamma^2\mu^r$ bad $r$-cliques. 
    
    \smallskip
    \noindent\emph{Step 1. Few edges off the partition.} Let $E_b$ be the
    set of edges of $G$ that lie inside a part or are incident to
    $V(G)\setminus C$, i.e.\ that are not contained in any
    $G[V_i,V_j]$ with $i\ne j$. Suppose that $|E_b|>\gamma m$. Then the
    $r$-partite subgraph $G':=\bigcup_{i<j}G[V_i,V_j]$ has
    fewer than $(1-\gamma)m$ edges. As $G'$ is
    $K_{r+1}$-free and every good $r$-clique of $G$ is an
    $r$-clique of $G'$, Theorem \ref{thm-SS-FR} yields at most
    $\big(e(G')/\binom{r}{2}\big)^{r/2}
    <(1-\gamma)^{r/2}\mu^r$ good $r$-cliques, and hence 
    \[ k_r(G)<(1-\gamma)^{r/2}\mu^r+\gamma^2\mu^r. \]
    This contradicts the hypothesis 
    $k_r(G)\ge (1-\beta)\mu^r\ge (1-\gamma^4)\mu^r$, because 
    $1-\gamma^4>\gamma^2+(1-\gamma)^{r/2}$ for every 
    $0<\gamma<\tfrac12$. Hence, we obtain $|E_b|\le\gamma m$.

    \smallskip
    \noindent\emph{Step 2. Few missing cross-edges.} By the above (\ref{eq-count-bad-copy}),  we know that $G$ has
    at least $(1-\beta-\gamma^2)\mu^r$ good $r$-cliques. Let $E_s$
    be the set of edges of $T_{C,r}$ absent in $G$. Suppose that
    $|E_s|>\gamma m$. Each missing cross-edge between $V_i$ and
    $V_j$ lies in $\prod_{k\ne i,j}|V_k|\ge (|C|/r-1)^{r-2}$
    $r$-cliques of $T_{C,r}$, and each $r$-clique of $T_{C,r}$
    contains $\binom{r}{2}$ cross-edges; hence these missing edges
    destroy at least
    $\gamma m\,(|C|/r-1)^{r-2}/\binom{r}{2}$ distinct $r$-cliques.
    Therefore, the number of good $r$-cliques of $G$ is at most
    \[ N:= \left(\frac{|C|}{r}+1\right)^{\!r}
       -\frac{\gamma m}{\binom{r}{2}}
       \left(\frac{|C|}{r}-1\right)^{\!r-2}. \]
     On the one hand,
    $|C|\le\lceil r\mu\rceil$ gives $|C|/r \le\mu+1/r$, so
    $(|C|/r +1)^r\le \mu^r+O(\mu^{r-1})$. On the other hand, the
    number of good $r$-cliques of $G$ is at most 
    $\prod_{i=1}^r |V_i|\le (|C|/r+1)^r$, so
    $(|C|/r+1)^r\ge (1-\beta-\gamma^2)\mu^r$, which yields the lower bound 
    $|C|/r \ge (1-\beta-\gamma^2)^{1/r}\mu-1$.

    Using $m=\binom{r}{2}\mu^2$ and
    ${|C|}/{r}-1\ge (1-\beta-\gamma^2)^{1/r}\mu-2$,
    we get
    \[ \frac{\gamma m}{\binom{r}{2}}
       \left(\frac{|C|}{r}-1\right)^{r-2}
       \ge \gamma\left((1-\beta-\gamma^2)^{1/r}
       -\frac{2}{\mu}\right)^{r-2}\mu^r . \]
    Combining the two estimates,
    \[ N\le \mu^r+O(\mu^{r-1})-\gamma
       \left((1-\beta-\gamma^2)^{1/r}
       -\tfrac{2}{\mu}\right)^{r-2}\mu^r . \]
    Put $c:=(1-\beta-\gamma^2)^{(r-2)/r}$. As $m\to\infty$, we have $\mu\to\infty$ and
    $( (1-\beta-\gamma^2)^{1/r}-\tfrac{2}{\mu} )^{r-2}
    \to c$. Then
    $N\le \bigl(1-\gamma c+o(1)\bigr)\mu^r $. 
    Since $0<1-\beta-\gamma^2<1$ and 
    $(r-2)/r\le 1$, we have $c\ge 1-\beta-\gamma^2$, so $ \gamma c \ge \gamma(1-\beta-\gamma^2) > 
       \beta+\gamma^2 $, 
    where the last inequality holds since
    $0<\gamma<\tfrac14$ and $\beta\le\gamma^4$. 
    Hence $\gamma c-(\beta+\gamma^2)$ is a positive
    constant, so for $m$ large, the $o(1)$ term is dominated
    by it, giving
    \[ N\le \bigl(1-\gamma c+o(1)\bigr)\mu^r
       <(1-\beta-\gamma^2)\mu^r . \]
    This contradicts the lower bound $(1-\beta-\gamma^2)\mu^r$
    on the number of good $r$-cliques, so $|E_s|\le\gamma m$. 

    \smallskip
    Finally, we obtain $d(G,T_{C,r})\le |E_b|+|E_s|\le 2\gamma m
    =\varepsilon m$, as desired.
\end{proof}

Combining the lemmas above, we give an alternative proof of 
Theorem \ref{thm-stability-clique-4}. 

\begin{proof}[{\bf Second proof of Theorem \ref{thm-stability-clique-4}}]
Given arbitrary $\varepsilon>0$, let
$\beta$ and $m_0$ be as in Lemma \ref{lem-reduce-general}, and let
$\delta$ be as in Lemma \ref{lem-kr-lower} for this $\beta$. Take
the constant in Theorem \ref{thm-stability-clique-4} to be
$ \delta_{\ref{thm-stability-clique-4}}:=(1-\frac{1}{r}) \delta$. 

If $G$ is a $K_{r+1}$-free graph with
$m\ge m_0$ edges and
$\lambda^2(G)\ge \big(1-\tfrac1r- \delta_{\ref{thm-stability-clique-4}} \big)2m
$, then
\[ \lambda(G)\ge \sqrt{(1-\delta)\Big(1-\frac1r\Big)2m}
   \ge (1-\delta)\sqrt{\Big(1-\frac1r\Big)2m}
   =(1-\delta)(r-1)\mu. \]
By Lemma \ref{lem-kr-lower}, $k_r(G)\ge (1-\beta)\mu^r$, and by
Lemma \ref{lem-reduce-general}, $d(G,T_{C,r})\le\varepsilon m$
for some $C \subseteq V(G)$. \end{proof}

\end{document}